\documentclass[12pt]{amsart}
\usepackage{amsmath, amsthm, amsfonts, amssymb, color}
 \usepackage{mathrsfs}
\usepackage{amsfonts, amsmath}
 \usepackage{amsmath,amstext,amsthm,amssymb,amsxtra}
 \usepackage{txfonts} 
 \usepackage[colorlinks, citecolor=blue,pagebackref,hypertexnames=false]{hyperref}
 \allowdisplaybreaks
 \usepackage{pgf,tikz}
\begin{document}
\baselineskip 16pt

\newcommand\C{{\mathbb C}}
\newtheorem{theorem}{Theorem}[section]
\newtheorem{proposition}[theorem]{Proposition}
\newtheorem{lemma}[theorem]{Lemma}
\newtheorem{corollary}[theorem]{Corollary}
\newtheorem{remark}[theorem]{Remark}
\newtheorem{example}[theorem]{Example}
\newtheorem{question}[theorem]{Question}
\newtheorem{exercise}[theorem]{Exercise}
\newtheorem{definition}[theorem]{Definition}
\newtheorem{conjecture}[theorem]{Conjecture}
\newcommand\RR{\mathbb{R}}
\newcommand{\la}{\lambda}
\def\RN {\mathbb{R}^n}
\newcommand{\norm}[1]{\left\Vert#1\right\Vert}
\newcommand{\abs}[1]{\left\vert#1\right\vert}
\newcommand{\set}[1]{\left\{#1\right\}}
\newcommand{\Real}{\mathbb{R}}
\newcommand{\supp}{\operatorname{supp}}
\newcommand{\card}{\operatorname{card}}
\renewcommand{\L}{\mathcal{L}}
\renewcommand{\P}{\mathcal{P}}
\newcommand{\T}{\mathcal{T}}
\newcommand{\A}{\mathbb{A}}
\newcommand{\K}{\mathcal{K}}
\renewcommand{\S}{\mathcal{S}}
\newcommand{\blue}[1]{\textcolor{blue}{#1}}
\newcommand{\red}[1]{\textcolor{red}{#1}}
\newcommand{\Id}{\operatorname{I}}
\newcommand\wrt{\,{\rm d}}
\def\SH{\sqrt {H}}

\newcommand{\rn}{\mathbb R^n}
\newcommand{\de}{\delta}
\newcommand{\tf}{\tfrac}
\newcommand{\ep}{\epsilon}
\newcommand{\vp}{\varphi}

\newcommand{\mar}[1]{{\marginpar{\sffamily{\scriptsize
        #1}}}}

\newcommand\CC{\mathbb{C}}
\newcommand\NN{\mathbb{N}}
\newcommand\ZZ{\mathbb{Z}}
\renewcommand\Re{\operatorname{Re}}
\renewcommand\Im{\operatorname{Im}}
\newcommand{\mc}{\mathcal}
\newcommand\D{\mathcal{D}}
\newcommand{\al}{\alpha}
\newcommand{\nf}{\infty}
\newcommand{\comment}[1]{\vskip.3cm
	\fbox{%
		\color{red}
		\parbox{0.93\linewidth}{\footnotesize #1}}
	\vskip.3cm}

\newcommand{\disappear}[1]

\numberwithin{equation}{section}
\newcommand{\chg}[1]{{\color{red}{#1}}}
\newcommand{\note}[1]{{\color{green}{#1}}}
\newcommand{\later}[1]{{\color{blue}{#1}}}
\newcommand{\bchi}{ {\chi}}

\numberwithin{equation}{section}
\newcommand\relphantom[1]{\mathrel{\phantom{#1}}}
\newcommand\ve{\varepsilon}  \newcommand\tve{t_{\varepsilon}}
\newcommand\vf{\varphi}      \newcommand\yvf{y_{\varphi}}
\newcommand\bfE{\mathbf{E}}
\newcommand{\ale}{\text{a.e. }}

 \newcommand{\mm}{\mathbf m}
\newcommand{\Be}{\begin{equation}}
\newcommand{\Ee}{\end{equation}}

\title[A weighted $L^2$ estimate of Commutators for Hermite operator]
{A weighted $L^2$ estimate of Commutators of Bochner-Riesz Operators  for Hermite operator}

\author[Peng Chen and  Xixi Lin]{Peng Chen and  Xixi Lin}
\thanks{Corresponding author: Xixi Lin}
\address{Peng Chen, Department of Mathematics, Sun Yat-sen
	University, Guangzhou, 510275, P.R. China}
\email{chenpeng3@mail.sysu.edu.cn}
\address{Xixi Lin, Department of Mathematics, Sun Yat-sen   University,
	Guangzhou, 510275, P.R. China}
\email{linxx58@mail2.sysu.edu.cn}


\begin{abstract}
%
%
Let $H$ be the Hermite operator $-\Delta +|x|^2$ on $\mathbb{R}^n$. We prove a weighted $L^2$ estimate of the maximal commutator operator
$\sup_{R>0}|[b, S_R^\lambda(H)](f)|$, where
$
[b, S_R^\lambda(H)](f) = bS_R^\lambda(H) f - S_R^\lambda(H)(bf)
$
is the commutator of a BMO function  $b$ and the Bochner-Riesz means
$S_R^\lambda(H)$ for the Hermite operator $H$. As  an application, we obtain the almost everywhere convergence of $[b, S_R^\lambda(H)](f)$ for large $\lambda $  and  $f\in L^p(\mathbb{R}^n)$.	
\end{abstract}

\maketitle
\section{Introduction}
\setcounter{equation}{0}

Let $H$ denote the Hermite operator
   \begin{eqnarray}\label{h1}
 -\Delta + |x|^2 =-\sum_{i=1}^n {\partial^2\over \partial x_i^2} + |x|^2, \quad x=(x_1, \cdots, x_n), \ \ \  n\geq 1.
\end{eqnarray}
The operator $H$ is non-negative and self-adjoint with respect to the Lebesgue measure
on $\RN$. The spectrum of the operator $H$ is given by the set $ {2\mathbb N_0} + n$. Here ${\mathbb N_0}$ denotes the set of nonnegative integers.
 For each non-negative integer $k$, the Hermite polynomials $H_k(t) $ on $\RR$ are
defined by $H_k(t)=(-1)^k e^{t^2} {d^k\over d t^k} \big(e^{-t^2}\big)$, and  the Hermite functions
$h_k(t):=(2^k k !  \sqrt{\pi})^{-1/2} H_k(t) e^{-t^2/2}$, $k=0, 1, 2, \ldots$ form an orthonormal basis
of $L^2(\mathbb R)$.
For
any multiindex $\mu\in {\mathbb N}^n_0$,
the $n$-dimensional Hermite functions are given by tensor product of the one dimensional Hermite functions:
\begin{eqnarray}\label{ephi}
	\Phi_{\mu}(x)=\prod_{i=1}^n h_{\mu_i}(x_i), \quad \mu=(\mu_1, \cdots, \mu_n).
\end{eqnarray}
Then the functions $\Phi_{\mu}$ are eigenfunctions for
the Hermite operator with eigenvalue $(2|\mu|+n)$ and $\{\Phi_{\mu}\}_{\mu\in \mathbb N_0^n}$ form a complete orthonormal
system in $L^2({\RN})$.
Thus, for  every $f\in L^2(\RN)$  we have  the Hermite expansion
\begin{eqnarray} \label{e1.3}
	f(x)=\sum_{\mu\in \NN_0^n}\langle f, \Phi_{\mu}\rangle \Phi_\mu(x)=\sum_{k=0}^{\infty}P_kf(x),
\end{eqnarray}
where $P_k$ denotes the Hermite projection operator given by
\begin{eqnarray} \label{e1.5}
	P_kf(x)=\sum_{2|\mu|+n=k}\langle f, \Phi_{\mu}\rangle\Phi_\mu(x).
\end{eqnarray}
For $R>0$ the  Bochner-Riesz means for $H$ of order $\lambda\geq 0$   are defined by
\begin{eqnarray}\label{e1.4}
	S_R^{\lambda}(H)f(x)
	=
	\sum_{k=0}^{\infty} \left(1-{2k+n\over R^2}\right)_+^{\lambda} P_k f(x).
\end{eqnarray}
The assumption $\lambda\geq 0$ is necessary for $S_R^{\lambda}(H)$  to be   defined for all  $R>0$. Note that
	$S_R^{\lambda}(H)f$ can not be defined with $R^2=2k+n$ if  $\lambda<0$.

On the space $\mathbb{R}$,
Thangavelu \cite{T1} showed that $S_R^\lambda(H)$ is uniformly bounded on $L^p(\mathbb{R})$ for $1\leq p\leq \infty$
provided $\lambda>1/6$; if $0<\lambda<1/6$,  the uniformly boundedness of $S_R^\lambda(H)$  holds if and only if
$4/(6\lambda+3) <p<4/(1-6\lambda)$. On the space $\mathbb{R}^n$ for dimension $n\geq2$, if $\lambda> (n-1)/2$, Thangavelu \cite{T2} showed that $S_R^\lambda(H)$ is uniformly bounded on $L^p(\mathbb{R}^n)$ for $1\leq p\leq \infty$.
Let $n\geq2,\  0\leq\lambda\leq (n-1)/2$ and $p\in[1,\infty]\backslash\, \{2\}$, it was conjectured (see \cite[p.259]{T3}) that $S_R^\lambda(H)$ is bounded on $L^p(\mathbb{R}^n)$ uniformly in $R$ if and only if
$$\lambda>\lambda(p)=\max\Bigg\{n\left|\frac1p-\frac12\right|-\frac12,0\Bigg\}.$$
Thangavelu showed that the $L^p$ boundedness of $S_R^\lambda(H)$ fail if  $\lambda<\lambda(p)$.
Karadzhov \cite{K} showed  the $L^p$ boundedness of $S_R^\lambda(H)$  by an optimal $L^2$–$L^p$ spectral projection estimate when $p$ is in the range of $[1,2n/(n+2)]\cup[2n/(n-2),\infty]$ and $\lambda>\lambda(p)$.
Recently, Lee and Ryu \cite{LR} invalidated the above conjecture by showing that $\sup_{R>0}\|S_R^\lambda(H)\|_{L^p\rightarrow L^p}\leq C$ only if $\lambda\geq-1/(3p)+n/3(1/2-1/p)>\lambda(p)$ when $p\in(2(n+1)/n,2(2n-1)/(2n-3))$ for $n\geq2$.
 Concerning the estimate of maximal operator,  it is known (see \cite{CLSY}) that the maximal operator $\sup_{R>0}|S_R^\lambda(H)f|$ is bounded on $L^p(\RR^n)$ for $n\geq2$ whenever $p\geq2n/(n-2)$ and $\lambda>\lambda(p)$.   Further,
 the first author, Duong, He, Lee and Yan \cite{CDHLY}  proved that $\sup_{R>0}|S_R^\lambda(H)f|$ is bounded on $L^2(\RR^n,(1+|x|)^{-\alpha})$ if $\lambda>\max\{(\alpha-1)/4,0\},$ which implied the a.e. convergence of $S^\lambda_R(H)f$ for all $f\in L^p(\RR^n)$ provided that $\lambda>\lambda(p)/2$ with $p\in[2,\infty)$.

Consider the commutator, given an operator $T$ and  a local integral function $b$, the commutator of $T$ and $b$ is defined as follow
  $$[b,T]f(x):=bTf(x)-T(bf)(x).$$
It's well known that Coifman, Rochberg and Weiss \cite{CRW}  characterized the boundedness of the commutator $[b,T]$
with Riesz transforms and $b\in \mathrm{BMO}$.
Since then, lots of investigation had came out of this work:
generalizations to space of homogeneous type space by Uchiyama~\cite{U2}; multi-parameter extensions by Ferguson and  Lacey~\cite{FL} and by Lacey, Petermichl, Pipher and Wick~\cite{LPPW1}; in two weight setting by Holmes, Lacey and Wick~\cite{HLW}; $L^p$ to $L^q$ boundedness and  application with Jocobian operator by Hyt\"onen~\cite{H};  div-curl lemmas by Coifman, Lions, Meyer and Semmes~\cite{CLMS} and by Lacey, Petermichl, Pipher and Wick~\cite{LPPW2}; additional interpretations in operator theory by Uchiyama~\cite{U1} and by Nazarov, Pisier, Treil and Volberg~\cite{NPTV}; commutators with classical Bochner-Riesz operators by \'{A}lvarez, Bagby, Kurtz and P\'{e}rez~\cite{ABKP} and by Hu and Lu~\cite{HL1,HL3}.

In \cite{CLY}, the authors of this article and   Yan   studied   the $L^p$-boundedness of  the commutator $[b, S_R^\lambda(H)](f)$ of a BMO function  $b$ and the Bochner-Riesz means
$S_R^\lambda(H)$, which is defined by
$$
[b, S_R^\lambda(H)](f) = bS_R^\lambda(H) f - S_R^\lambda(H)(bf).
$$
 They showed that  if $n\geq2$,  $1\leq p\leq2n/(n+2)$ and $\lambda>\lambda(p)$, then for all $b\in \mathrm{BMO(\mathbb{R}^n)}$ and all $q\in(p,p')$,
	$$\sup_{R>0} \big\| \big[b,S_R^{\lambda}(H)\big ]  \big\|_{q\rightarrow q}\leq C\|b\|_{\mathrm{BMO}}.
	$$
 The purpose of this paper is to follow this line to
  establish  the weighted $L^2$ estimates of commutator of $S_R^\lambda(H)$ and a BMO  function $b$. Our main result is the following.

\begin{theorem}\label{thm1.1}
  Let $b\in\mathrm{BMO}(\mathbb{R}^n)$. For $0\leq \alpha<n$, if $\lambda>\max\{(\alpha-1)/4,0\}$, then
  \begin{align}\label{e1}
    \int_{\mathbb{R}^n} \sup_{R>0}| [b,S^{\lambda}_R(H)]f(x)|^2(1+|x|)^{-\alpha}\mathrm{d}x\leq C(n,\alpha,\lambda)\|b\|_{\mathrm{BMO}}^2 \int_{\mathbb{R}^n} |f(x)|^2(1+|x|)^{-\alpha}\mathrm{d}x.
  \end{align}
\end{theorem}

As a consequence of Theorem~\ref{thm1.1}, we have the following result.

\begin{corollary}\label{coro1.2}
Let $2\leq p<\infty$ and $\lambda>\lambda(p)/2.$  Then for any $b\in\mathrm{BMO}(\mathbb{R}^n)$ and  $f\in L^p(\mathbb{R}^n)$,
$$\lim_{R\rightarrow\infty}[b,S^{\lambda}_R(H)]f(x)=0
$$
 almost everywhere.
\end{corollary}

We would like to mention that the classical Bochner-Riesz means on $\RR^n$ is defined by
 \begin{eqnarray*}
 	\widehat{S^{\lambda}_Rf}(\xi)
 	=\left(1-{|\xi|^2\over R^2}\right)_+^{\lambda} \widehat{f}(\xi),  \quad \forall{\xi \in \RN}.
 \end{eqnarray*}
 Hu and Lu \cite{HL2} showed that for    $\lambda>0$,  the maximal commutator operator $\sup_{R>0}|[b,S^\lambda_R]f|$ is bounded on $L^2(\RR^n)$.
  Further they proved
  a weighted estimate that  the maximal commutator operator $\sup_{R>0}|[b,S^\lambda_R]f|$ is bounded on $L^2(\RR^n,|x|^{-\alpha})$ whenever $0<\alpha<n$ and $\lambda>\max\{(\alpha-1)/2,0\}$ in \cite{HL3}.

  The proof of Theorem~\ref{thm1.1} relies on a weighted $L^2$ estimate for the square function $  G_{b,\delta}$ which is defined by
  $$
  G_{b,\delta}(f)(x):=\left(\int_{0}^{\infty}\left|\, \left[b,\phi\left(\delta^{-1}\left( 1- {H\over t^2}\right) \right) \right] f(x)\right |^2\frac{\mathrm{d}t}{t}\right)^{\frac12},
  $$
 where  $\phi\in C_c^\infty(\RR)$ with support $\{x:1/8\leq |x| \leq 1/2\}$ and  $|\phi|\leq 1.$
 (see Proposition~\ref{Pro1} below). Indeed, we will show that  for any $0<\upsilon\leq1/2$, there exists a constant $C_{\alpha,\upsilon}>0$  independent of $\delta$ such that
\begin{align}\label{e02}
	\int_{\mathbb{R}^n}|G_{b,\delta}(f)(x)|^2(1+|x|)^{-\alpha}\mathrm{d}x\leq C_{\alpha,\upsilon} \|b\|_{\mathrm{BMO}}^2B^{\upsilon}_{\alpha,n}(\delta)
	\int_{\mathbb{R}^n}|f(x)|^2(1+|x|)^{-\alpha}\mathrm{d}x,
\end{align}
where
\begin{align}\label{BBB}
	B^{\upsilon}_{\alpha,n}(\delta)=
	\begin{cases}
		\delta^{1-\upsilon}, & \mbox{\rm{if} } 0< \alpha< 1, n=1; \alpha=0, n\geq 1;\\[4pt]
		\delta^{\frac{3-\alpha}{2}-\upsilon}, & \mbox{\rm{if} } 1<\alpha <n, n\geq2.
	\end{cases}
\end{align}
 To show \eqref{e02},   we will use an extension of two non-trivial facts due to \cite{CDHLY}. The first   is that for  any  $\alpha\geq0$,
 \begin{align}\label{a3a}
 	\|(1+|x|)^{2\alpha}f\|_{L^2(\RR^n)}\leq C\|(I+H)^{\alpha}f\|_{L^2(\RR^n)}
 \end{align}
 holds for any $f\in {\mathscr S} (\RR^n)$.  The second fact  is
  a type of trace lemma for the Hermite operator, that is, for $\alpha>1$, there exists a constant $C>0$ such that
\begin{align}\label{a1a}
\|\chi_{[k,k+1)}(H)\|_{L^2(\RR^n)\rightarrow L^2(\RR^n,(1+|x|)^{-\alpha})}\leq Ck^{-1/4},\ \ k\in \NN^+.
\end{align}
We would like to mention that when $0<\alpha<1$, \eqref{a1a} is not applicable. To show the square function estimate~\eqref{e02} for  $0<\alpha<1$, we make use of a weighted Plancherel-type estimate(see the estimate~\eqref{e4} below and refer to \cite[Lemma~2.6]{CDHLY} for the proof).

This paper is organized as follows. In Section~$2$, we give some preliminary results about Hermite operator,  and some estimates of the commutator of spectral multipliers and BMO functions, which provide basic estimates required
for the proof of  Theorem~\ref{thm1.1}.
 We establish a weighted estimate~\eqref{e02} of the  square function $G_{b,\delta}$ in Section~$3$. The proof of Theorem~\ref{thm1.1} will be given in Section~$4$   by using the estimate~\eqref{e02} of the square function $G_{b,\delta}$.  As a consequence of Theorem~\ref{thm1.1}, we obtain the proof of Corollary~\ref{coro1.2} at the end of Section~$4$.

\medskip


\section{  Preliminary results }
\setcounter{equation}{0}

We start by recalling some properties of the Hermite operator $H$. The Hermite operator $H$ satisfies finite speed propagation property, i.e.,
\begin{align}\label{FS}
\mathrm{supp}\  K_{\cos(t\sqrt{H})}(x,y)\subseteq \mathfrak{D}_t:=\{(x,y):|x-y|\leq t\}.
\tag{\rm{FS}}
\end{align}
 See for example, \cite[Theorem 2]{S}. By Fourier inversion, for any  even function  $F$,
$$F(\sqrt{H})=\frac{1}{2\pi}\int_{-\infty}^{+\infty}\widehat{F}(t)\cos(t\sqrt{H})\mathrm{d}t.$$
From \cite[Lemma I.1]{COSY},  it tells us that if    supp$\widehat{F}\subseteq [-t,t]$, then
\begin{align}\label{d3}
K_{F(\sqrt{H})}(x,y)\subseteq \mathfrak{D}_t,
\end{align}
which will be used in the sequel.

For any function $F$ with support in $[-1,1]$ and $2\leq p<\infty$,  we define
\begin{align*}
  \|F\|_{N, _p}:=\left(\frac{1}{N}\sum_{i=-N+1}^{N}\sup_{\lambda\in\left[\frac{i-1}{N},\frac{i}{N}\right)}|F(\lambda)|^p\right)^{\frac{1}{p}},\ \ N\in \NN^+.
\end{align*}

The following is the trace lemma for Hermite operator.
\begin{lemma}\label{L1}
For $\alpha>1$, there exists a constant $C>0$ such that for any $k\in \NN^+$,
\begin{align}\label{e004}
\|\chi_{[k,k+1)}(H)\|_{L^2(\RR^n)\rightarrow L^2(\RR^n,(1+|x|)^{-\alpha})}\leq Ck^{-1/4}.
\end{align}
As a consequence,  for any function $F$ supported in $[N/4,N]$, $N\in\mathbb{N^+}$ and any $\varepsilon>0$, there exist constant $C$ and $C_\varepsilon$ such that
\begin{align}\label{e4}
\int_{\mathbb{R}^n}|F(\sqrt{H})f(x)|^2(1+|x|)^{-\alpha}\mathrm{d}x\leq
\begin{cases}
 \  CN\|F(N\cdot)\|_{N^2, _2}^2\|f\|_{L^2(\RR^n)}^2, & \mbox{\rm{if}  \  } \alpha>1; \\[4pt]
 \  C_{\varepsilon}N^{\frac{\alpha}{1+\varepsilon}}\|F(N\cdot)\|_{N^2,_{2(1+\varepsilon)/\alpha}}^2\|f\|_{L^2(\RR^n)}^2,
  & \mbox{\rm{if}  \  } 0<\alpha\leq1.
\end{cases}
\end{align}
\end{lemma}

\begin{proof}
For the proof of \eqref{e004} and \eqref{e4}, we refer the reader to \cite[Lemmas~1.5, 2.4 and  2.6]{CDHLY}.
 We would like to mention that estimate \eqref{e4} for $\alpha>1$ is equivalent to  estimate~\eqref{e004}.  Estimate \eqref{e4} for $0<\alpha\leq1$ is a consequence of a bilinear interpolation of estimate~\eqref{e4} for $\alpha>1$ and the trivial fact $\|F(\sqrt{H})f\|_{L^2(\RR^n)}=\|F(N\cdot)\|_{L^{\infty}(\RR)} \|f\|_{L^2(\RR^n)}$.
\end{proof}

\begin{lemma}\label{L2}
	Let $\alpha\geq0$.  Then the estimate
	\begin{align*}
		\|(1+|x|)^{\alpha/2}f\|_{L^2(\RR^n)}\leq C\|(I+H)^{\alpha/4}f\|_{L^2(\RR^n)},
	\end{align*}
	holds for any $f\in {\mathscr S}({\mathbb R^n})$. Here,
	${\mathscr S}({\mathbb R^n})$ stands for  the class of Schwartz functions in ${\mathbb R^n}.$
\end{lemma}

 \begin{proof}
For the proof, we refer the reader to \cite[Lemma 1.4]{CDHLY}.
\end{proof}

\begin{lemma}\label{L3}
Let $b\in\mathrm{BMO}(\mathbb{R}^n)$. Denote $M_b$ the commutator of the Hardy-Littlewood maximal operator defined by
  \begin{align*}
    M_b(f)(x):=\sup_{r>0}r^{-n}\int_{|x-y|<r}|(b(x)-b(y))f(y)| \mathrm{d}y.
  \end{align*}
If $1<p<\infty$ and $w\in A_p$, then $M_b$ is bounded on $L^p(\mathbb{R}^n,w)$ with bound $C(n,p)\|b\|_{\mathrm{BMO}}$.
\end{lemma}

 \begin{proof}
For the proof, we refer the reader to \cite[Lemma~1]{HL3} and \cite{GHST}.
\end{proof}

\begin{lemma}\label{L4}
   Let $M_b(f)$ be defined as above and $\varphi\in C_c^{\infty}({\mathbb{R}})$. Then for any $\varepsilon>0$,
   \begin{align}\label{ee1}
     \sup_{t>0}|\,[b,\varphi(t^{-2}H)]f(x)|\leq C_{\varepsilon}\|\varphi\|_{W^2_{n+1/2+\varepsilon}} M_b(f)(x).
   \end{align}
   In addition, for any $1<p<\infty$ and $w\in A_p$,
   \begin{align}\label{ee2}
   \|\sup_{t>0}|\,[b,\varphi(t^{-2}H)]f|\,\|_{L^p(\mathbb{R}^n,w)}\leq C_{\varepsilon}\|\varphi\|_{W^2_{n+1/2+\varepsilon}} \|f\|_{L^p(\mathbb{R}^n,w)}.
   \end{align}

\end{lemma}
\begin{proof}
  Let $G(t^{-2}H)=\varphi(t^{-2}H)e^{t^{-2}H}$, then by the Fourier transform, we have
  $$\varphi(t^{-2}H)=\int_{\mathbb{R}} \widehat{G}(\tau) e^{-t^{-2}(1-i\tau)H}\mathrm{d}\tau,$$
  with the kernel
  $$K_{\varphi(t^{-2}H)}(x,y)=\int_{\mathbb{R}}\widehat{G}(\tau)p_{t^{-2}(1-i\tau)}(x,y)\mathrm{d}\tau,$$
  where $p_{t}(x,y)$ is the heat kernel of the semigroup $e^{-tH}$. Then
  \begin{align}\label{ee3}
    |[b,\varphi(t^{-2}H)]f(x)|&=|\int_{\mathbb{R}^n}K_{\varphi(t^{-2}H)}(x,y)(b(x)-b(y))f(y)\mathrm{d}y|\nonumber\\
   &\leq \int_{\mathbb{R}}|\widehat{G}(\tau)|\int_{\mathbb{R}^n}|p_{t^{-2}(1-i\tau)}(x,y)(b(x)-b(y))f(y)|\mathrm{d}y\mathrm{d}\tau.
  \end{align}
  The kernel of  $e^{-t^{-2}H}$ has the Gaussian upper bound. Let $z=t^{-2}(1-i\tau)$. By the Phragmen–Lindel\"of Theorem, the kernel of $e^{-zH}$ satisfies the following estimate(see \cite[Theorem~7.2]{Ou}),
\begin{align}\label{ee3_1}
  |p_{z}(x,y)|\leq
  C|z|^{-n/2}(1+|\tau|^2)^{n/4}\exp\left(-c\frac{|x-y|^2}{|z|(1+|\tau|^2)^{1/2}}\right).
\end{align}
  Let $r=t^{-1}(1+|\tau|^2)^{1/2}$, $U_0(B)=B(x,r)$, $U_k(B)=2^{k}B-2^{k-1}B$ for $k\geq1$. By estimate~\eqref{ee3_1},
  \begin{align}\label{ee4}
  \int_{\mathbb{R}^n}|p_{z}(x,y)(b(x)-b(y))f(y)|\mathrm{d}y
 &\leq C\sum_{k\geq 0}\frac{\exp(-c2^{2k})}{(1+|\tau|^2)^{-n/2}r^n}\int_{U_k(B)}|b(x)-b(y)||f(y)|\mathrm{d}y\nonumber \\
 &\leq  C\sum_{k\geq 0}\frac{2^{kn}\exp(-c2^{2k})}{(1+|\tau|^2)^{-n/2}|2^kB|}\int_{2^{k}B}|b(x)-b(y)||f(y)|\mathrm{d}y\nonumber \\
 &\leq  C\sum_{k\geq 0}\frac{2^{kn}\exp(-c2^{2k})}{(1+|\tau|^2)^{-n/2}} M_b(f)(x)\nonumber \\
 &\leq  C(1+|\tau|^2)^{n/2}M_b(f)(x).
  \end{align}
Note that  $\varphi\in\mathbb{R}$ has compact support, $\|\varphi\|_{W^2_{n+1/2+\varepsilon}}\approx\|G\|_{W^2_{n+1/2+\varepsilon}}
$. This, in combination with  estimates~\eqref{ee3} and \eqref{ee4},  implies that
  \begin{align}\label{ee7}
 |[b,\varphi(t^{-2}H)]f(x)|\leq\int_{\mathbb{R}}|\widehat{G}(\tau)|(1+|\tau|^2)^{n/2}M_b(f)(x)\mathrm{d}\tau\leq C_{\varepsilon}\|\varphi\|_{W^2_{n+1/2+\varepsilon}} M_b(f)(x).
  \end{align}
 Finally, the $L^p(\mathbb{R}^n,w)$ of $\sup_{t>0}|[b,\varphi(t^{-2}H)]f|$ follows by  estimate~\eqref{ee7} and Lemma~\ref{L3}.
\end{proof}

\begin{lemma}\label{lem55}
Let $b\in \mathrm{BMO}(\mathbb{R}^n)$, $s>n/2$ and let $r_0=\max\{1,n/s\}$.  Then for all   Borel function $F$  such that $\sup_{R>0}\|\eta F(R\cdot)\|_{W_s^{\infty}}<\infty$ where $\eta\in C_c^{\infty}(0,\infty)$ is a fixed function  and not identically zero,
the commutator $[b,F(\sqrt{H})]$ is bounded on $L^p(\RR^n,w)$ for all $r_0<p<\infty$ and  $w\in A_{p/{r_0}}$. In addition,
  \begin{align*}
\|[b,F(\sqrt{H})]\|_{L^p(\RR^n,w)\rightarrow L^p(\RR^n,w)}\leq C\|b\|_{\mathrm{BMO}}(\sup_{R>0}\|\eta F(R\cdot)\|_{W^\infty_s}+|F(0)|).
\end{align*}
\end{lemma}
\begin{proof}
The Hermite operator has Gaussian upper bound, thus it satisfies the condition of \cite[Theorem~1.1]{B}, which gives this lemma. See also \cite[Theorem~3.2]{DSY}.

\end{proof}

\begin{lemma}\label{L6}
  Let a non-zero function $\varphi\in C_c^{\infty}(\mathbb{R})$ with support $\{u:1\leq |u|\leq 3\}$.  For  $-n<\alpha<n$,
   \begin{align}\label{ee5}
    \int_{\mathbb{R}^n} \sum_{k\in\mathbb{Z}}|\varphi(2^{-k}\sqrt{H})f(x)|^2 (1+|x|)^{-\alpha} \mathrm{d}x\leq
    C \int_{\mathbb{R}^n} |f(x)|^2(1+|x|)^{-\alpha} \mathrm{d}x,
      \end{align}
  and for any $b\in \mathrm{BMO}(\mathbb{R}^n)$,
  \begin{align}\label{ee6}
    \int_{\mathbb{R}^n} \sum_{k\in\mathbb{Z}}|[b,\varphi(2^{-k}\sqrt{H})]f(x)|^2 (1+|x|)^{-\alpha} \mathrm{d}x\leq
    C \|b\|_{\mathrm{BMO}}^2\int_{\mathbb{R}^n} |f(x)|^2(1+|x|)^{-\alpha} \mathrm{d}x.
  \end{align}
\end{lemma}

\begin{proof}
For the proof of
 \eqref{ee5}, we refer it to \cite[Proposition~2.7]{CDHLY}. We  show the proof of estimate \eqref{ee6} for completeness, although the proof is rather standard. Indeed, let $r_k(t)$ be the Rademacher functions and $\varphi_k(\lambda)=\varphi(2^{-k}\lambda)$. Define a function
  $$F(t,\lambda):=\sum_{k\in\ZZ}r_k(t)\varphi_k(\lambda).$$
  By the property of Rademacher functions, we have
  \begin{align*}
    \sum_{k\in\ZZ}|[b,\varphi_k(\sqrt{H})]f(x)|^2 \leq C\int_{0}^{1}|\sum_{k\in\ZZ}r_k(t)[b,\varphi_k(\sqrt{H})]f(x)|^2\mathrm{d}t=C
    \int_{0}^{1}|\,[b,F(t,\sqrt{H})]f(x)|^2\mathrm{d}t.
  \end{align*}
Integrating in $x$ with weight $(1+|x|)^{-\alpha}$, we see that
  \begin{align*}
    \int_{\mathbb{R}^n}\sum_{k\in\ZZ}|[b,\varphi_k(\sqrt{H})]f(x)|^2(1+|x|)^{-\alpha} \mathrm{d}x\leq
    C\int_{0}^{1}\int_{\mathbb{R}^n}|[b,F(t,\sqrt{H})]f(x)|^2(1+|x|)^{-\alpha} \mathrm{d}x\mathrm{d}t.
  \end{align*}
 Let $\eta\in C_c^\infty(\mathbb{R^+})$. It's easily to obtain that $\sup_{R>0}\|\eta F(t,R\cdot)\|_{W^\infty_{s}(\RR)}<\infty$ for $s>n/2$. It follows by  Lemma~\ref{lem55} and $(1+|x|)^{-\alpha}\in A_2$ whenever $-n<\alpha<n$ that
  \begin{align*}
  \int_{\mathbb{R}^n}|[b,F(t,\sqrt{H})]f(x)|^2(1+|x|)^{-\alpha} \mathrm{d}x
  &\leq C\|b\|^2_{\mathrm{BMO}}\sup_{R>0}\|\eta F(t,R\cdot)\|_{W^\infty_{s}(\RR)}\int_{\mathbb{R}^n} |f(x)|^2(1+|x|)^{-\alpha} \mathrm{d}x\\
  &\leq C\|b\|^2_{\mathrm{BMO}}\int_{\mathbb{R}^n} |f(x)|^2(1+|x|)^{-\alpha} \mathrm{d}x,
  \end{align*}
  with $C$ uniformly in $t\in[0,1]$.
\end{proof}


\section{ A weighted estimate for the square function}

In this section, we will show the following  weighted $L^2$ estimates for the square function $  G_{b,\delta}$, which is defined by
$$
G_{b,\delta}(f)(x):=\left(\int_{0}^{\infty}\left|\, \left[b,\phi\left(\delta^{-1}\left( 1- {H\over t^2}\right) \right) \right] f(x)\right |^2\frac{\mathrm{d}t}{t}\right)^{\frac12},
$$
where  $\phi\in C_c^\infty(\RR)$ with support $\{x:1/8\leq |x| \leq 1/2\}$ and $|\phi|\leq 1.$

\begin{proposition}\label{Pro1}
  Let $0\leq\alpha<n$, $0<\delta< 1/2$. Assume $b\in \mathrm{BMO}(\RR^n)$.
   Then for any $0<\upsilon\leq1/2$, there exists a constant $C_{\alpha,\upsilon}>0$  independent of $\delta$ such that
  \begin{align}\label{e2}
    \int_{\mathbb{R}^n}|G_{b,\delta}(f)(x)|^2(1+|x|)^{-\alpha}\mathrm{d}x\leq C_{\alpha,\upsilon} \|b\|_{\mathrm{BMO}}^2B^{\upsilon}_{\alpha,n}(\delta)
    \int_{\mathbb{R}^n}|f(x)|^2(1+|x|)^{-\alpha}\mathrm{d}x,
  \end{align}
  where
  \begin{align*}
  B^{\upsilon}_{\alpha,n}(\delta)=
  \begin{cases}
    \delta^{1-\upsilon}, & \mbox{\rm{if} } 0< \alpha< 1, n=1; \alpha=0,\  n\geq 1;\\[4pt]
    \delta^{\frac{3-\alpha}{2}-\upsilon}, & \mbox{\rm{if} } 1<\alpha <n,\  n\geq2.
  \end{cases}
  \end{align*}
\end{proposition}

 We select an even function $\eta\in C_c^{\infty}(\mathbb{R})$ supported in $\{u:1/2\leq |u|\leq 2\}$ such that $\sum_{j\in\mathbb{Z}}\eta(2^{-j}u)=1,\forall\ u>0.$ Given $0<\delta<1/2$, let $j_0=[-\log_2\delta]-1$. Set $\eta_j(u)=\eta(2^{-j}u)$ for $j>j_0$ and $\eta_{j_0}(u)=1-\sum_{j\geq j_0+1}\eta(2^{-j}u)$, then we have
$\sum_{j\geq j_0}\eta_j(u)\equiv1,\forall u>0.$
Let us use $\phi_{\delta}(s)$ to denote $\phi(\delta^{-1}(1-s^2))$. For $j\geq j_0$, we define
\begin{align}\label{d5}
  \phi_{\delta,j}(s) =\frac{1}{2\pi}\int_{\mathbb{R}}\eta_j(u)\widehat{\phi_{\delta}}(u)e^{isu}\mathrm{d}u.
\end{align}
Following from dyadic decomposition, we have
\begin{align}\label{d0}
\phi_{\delta}(t^{-1}s)=\sum_{j\geq j_0}\phi_{\delta,j}(t^{-1}s).
\end{align}
The following  is a useful  estimate through the paper, for any $N\in \mathbb{N}$ and $j\geq j_0$,
\begin{align}\label{d1}
  |\phi_{\delta,j}(s)|\leq
  \begin{cases}
    C_N2^{(j_0-j)N}, & \mbox{if } |s|\in[1-2\delta,1+2\delta];\\
   C_N 2^{j-j_0}\big(1+2^j|\,s-1|\big)^{-N}, & \mbox{otherwise }.
  \end{cases}
\end{align}
See \cite[p.23 (3.16)]{CDHLY} for the proof.

To prove the proposition~\ref{Pro1}, we need the following lemmas.

\begin{lemma}\label{L7}
Let $0\leq\alpha<n$, $k\geq 0$ and $j\geq j_0$. We define an operator associated with $\phi_{\delta,j}$ by
$$T_{j,k}^{\delta}(f):=\left(\int_{2^{k-1}}^{2^{k+2}} |\phi_{\delta,j}(t^{-1}\sqrt{H})f|^2\frac{\mathrm{d}t}{t}\right)^{1/2}.$$
Then for any $0<\varepsilon\leq1/2$ and  $N\in\mathbb{N}$, there exists a constant $C_{\varepsilon,N}$ such that
   \begin{align*}
 \left\|T_{j,k}^{\delta}(f)\right\|_{L^2(\RR^n,(1+|x|)^{-\alpha})}\leq C_{\varepsilon,N}2^{(j_0-j)N}A^{\varepsilon}_{n}(\delta,\alpha)^{1/2}\left\|f\right\|_{L^2(\RR^n,(1+|x|)^{-\alpha})},
  \end{align*}
  where
    \begin{align}\label{AAA}
  A^{\varepsilon}_{n}(\delta,\alpha):=
  \begin{cases}
    \delta,&  \mbox{\rm{if} } \alpha=0, n\geq 1;\\
    \delta^{1-\varepsilon}, & \mbox{\rm{if} } 0<\alpha<1 , n=1 ;\\
    \delta^{\frac{3-\alpha}{2}}, & \mbox{\rm{if} } 1< \alpha<n, n\geq2.
  \end{cases}
  \end{align}
\end{lemma}

\begin{lemma}\label{L8}
 Let $T_{j,k}^{\delta}$ be defined in Lemma~\ref{L7} and $k\geq 0$, $0\leq\alpha< n$, $j\geq j_0$. Then for any $0<\varepsilon\leq1/2$  and  $N\in\mathbb{N}$, there exists a constant $C_{\varepsilon,N}$ such that
  \begin{align*}
 \|T_{j,k}^{\delta}(f)\|_{L^r(\mathbb{R}^n,(1+|x|)^{-\alpha})}
 \leq C_{\varepsilon,N}2^{(j_0-j)N}2^{(1-\theta) kn}A^{\varepsilon}_{n}(\delta,\alpha)^{\theta/2}\|f\|_{L^{r'}(\mathbb{R}^n,(1+|x|)^{-\alpha(r'-1)})},
  \end{align*}
   where $\theta=2/r$, $2< r<\infty$.
\end{lemma}
The proof of  Lemmas~\ref{L7} and \ref{L8} will be given later.  Next, let us use
 Lemmas~\ref{L7} and \ref{L8}     to prove  Proposition~\ref{Pro1}.

\begin{proof}[Proof of Proposition~\ref{Pro1}]
The eigenvalue of the Hermite operator is bigger  than $1$ and supp$\,\phi\subseteq\{x:1/8\leq |x|\leq 1/2\}$ imply that
$$
  \int_{0}^{+\infty}\int_{\mathbb{R}^n} |[b,\phi_{\delta}(t^{-1}\sqrt{H})]f(x)|^2(1+|x|)^{-\alpha}\mathrm{d}x\frac{\mathrm{{d}}t}{t}
 = \int_{1/2}^{+\infty}\int_{\mathbb{R}^n} |[b,\phi_{\delta}(t^{-1}\sqrt{H})]f(x)|^2(1+|x|)^{-\alpha}\mathrm{d}x\frac{\mathrm{d}t}{t}.
$$

Choose a function $\varphi$ with support $\{s:1\leq |s|\leq3\}$ and $\sum_{k\in\mathbb{Z}}\varphi(2^{-k}s)=1,\forall s>0$. Let $\varphi_k(s)=\varphi(2^{-k}s)$.
There is a uniform bound $C$ for any $t\in [1/2,\infty)$ such that $\#\{k\in \ZZ:\phi_{\delta}(t^{-1}s)\varphi_k(s)\neq  0,s>0\}\leq C$, where $\#$ is the counting measure. Hence, we have
\begin{align*}
&\int_{1/2}^{+\infty}\int_{\mathbb{R}^n} |[b,\phi_{\delta}(t^{-1}\sqrt{H})]f(x)|^2(1+|x|)^{-\alpha}\mathrm{d}x\frac{\mathrm{d}t}{t}\nonumber\\
&\leq C\sum_{k\in\ZZ}\int_{1/2}^{\infty}\int_{\mathbb{R}^n}|[b,\phi_{\delta}(t^{-1}\sqrt{H})\varphi_k(\sqrt{H})]f(x)|^2 (1+|x|)^{-\alpha}\mathrm{d}x\frac{\mathrm{d}t}{t}\\
&=C\sum_{k\geq0}\int_{2^{k-1}}^{2^{k+2}}\int_{\mathbb{R}^n}|[b,\phi_{\delta}(t^{-1}\sqrt{H})\varphi_k(\sqrt{H})]f(x)|^2 (1+|x|)^{-\alpha}\mathrm{d}x\frac{\mathrm{d}t}{t},
\end{align*}
where the last equality we use the support property of $\phi_{\delta}(t^{-1}s)$ and $\varphi_k(s)$.

\hspace{-0.5cm} Note that
$$[b,\phi_{\delta}(t^{-1}\sqrt{H})\varphi_k(\sqrt{H})]f=
[b,\phi_{\delta}(t^{-1}\sqrt{H})]\varphi_k(\sqrt{H})f+\phi_{\delta}(t^{-1}\sqrt{H})[b,\varphi_k(\sqrt{H})]f.
$$
It follows that
\begin{align}\label{f0}
&\int_{0}^{+\infty}\int_{\mathbb{R}^n} |[b,\phi_{\delta}(t^{-1}\sqrt{H})]f(x)|^2(1+|x|)^{-\alpha}\mathrm{d}x\frac{\mathrm{d}t}{t}\leq C\left(\sum_{k\geq0}I_1(k)+\sum_{k\geq0}I_2(k)\right),
\end{align}
where
\begin{align*}
  I_1(k):=\int_{2^{k-1}}^{2^{k+2}}\int_{\mathbb{R}^n}|\phi_{\delta}(t^{-1}\sqrt{H})[b,\varphi_k(\sqrt{H})]f(x)|^2
 (1+|x|)^{-\alpha}\mathrm{d}x\frac{\mathrm{d}t}{t},
\end{align*}
\begin{align*}
  I_2(k):=\int_{2^{k-1}}^{2^{k+2}}\int_{\mathbb{R}^n}|[b,\phi_{\delta}(t^{-1}\sqrt{H})]\varphi_k(\sqrt{H})f(x)|^2
 (1+|x|)^{-\alpha}\mathrm{d}x\frac{\mathrm{d}t}{t}.
\end{align*}
For the first term $I_1(k)$, it follows by equality~\eqref{d0} and Minkowski's inequality that
    \begin{align*}
 I_1(k) &\leq \left(\sum_{j\geq j_0}\left( \int_{2^{k-1}}^{2^{k+2}}\int_{\mathbb{R}^n}
   |\phi_{\delta,j}(t^{-1}\sqrt{H})[b,\varphi_k(\sqrt{H})]f(x)|^2(1+|x|)^{-\alpha}\mathrm{d}x\frac{\mathrm{d}t}{t}
  \right)^{1/2} \right)^2.
  \end{align*}
 By Lemma~\ref{L7},  we have
    \begin{align*}
   I_1(k) & \leq C_{\varepsilon,N}\left(\sum_{j\geq j_0}2^{(j_0-j)N}\right)^2A^{\varepsilon}_{n}(\delta,\alpha)
   \int_{\mathbb{R}^n}|[b,\varphi_k(\sqrt{H})]f(x)|^2(1+|x|)^{-\alpha}\mathrm{d}x\\
   &\leq C_{\varepsilon}A^{\varepsilon}_{n}(\delta,\alpha)
   \int_{\mathbb{R}^n}|[b,\varphi_k(\sqrt{H})]f(x)|^2(1+|x|)^{-\alpha}\mathrm{d}x.
 \end{align*}
 Then taking the sum over $k$ and using estimate~\eqref{ee6} in Lemma~\ref{L6} give
 \begin{align}\label{f1}
  \sum_{k\geq0}I_1(k)\leq C_{\varepsilon}A^{\varepsilon}_{n}(\delta,\alpha)
   \int_{\mathbb{R}^n}|f(x)|^2(1+|x|)^{-\alpha}\mathrm{d}x.
 \end{align}
For the second term $I_2(k)$, it follows by equality~\eqref{d0} and Minkowski's inequality that
  \begin{align*}
 I_2(k) &\leq \left(\sum_{j\geq j_0}\left( \int_{2^{k-1}}^{2^{k+2}}\int_{\mathbb{R}^n}
   |[b,\phi_{\delta,j}(t^{-1}\sqrt{H})]\varphi_k(\sqrt{H})f(x)|^2(1+|x|)^{-\alpha}\mathrm{d}x\frac{\mathrm{d}t}{t}
  \right)^{1/2} \right)^2.
  \end{align*}
  Fix $j,k$. We decompose $\mathbb{R}^n$ into disjoint cubes with side length $2^{j-k+2}$.
  For a given  $\mathbf{m}=(\mathbf{m}_1,\cdots,\mathbf{m}_n)\in\mathbb{Z}^n$, we define an associated cube by
$$Q_\mathbf{m}=\Big[2^{j-k+2}(\mathbf{m}_1-1/2),2^{j-k+2}(\mathbf{m}_1-1/2)\Big)
\times\cdots\times\Big[2^{j-k+2}(\mathbf{m}_n-1/2),2^{j-k+2}(\mathbf{m}_n-1/2)\Big).$$
Clearly, $Q_\mathbf{m}$ is a cube with centre $(\mathbf{m}_1,\cdots,\mathbf{m}_n)$ and side length $2^{j-k+2}$. $\{Q_\mathbf{m}\}$ are disjoint and $\mathbb{R}^n=\bigcup_{\mathbf{m}\in\mathbb{Z}^n}Q_\mathbf{m}.$ For each $\mathbf{m}\in\mathbb{Z}^n$, we define $\widetilde{Q}_\mathbf{m}$ by
$$\widetilde{Q}_\mathbf{m}
:=\bigcup_{\mathbf{m}'\in \mathbb{Z}^n:\mathrm{dist}(Q_\mathbf{m},Q_{\mathbf{m}'})\leq\sqrt{n}2^{j-k+3}}Q_{\mathbf{m}'}.$$
If $t\in[2^{k-1},2^{k+2}]$, supp~$\widehat{\phi_{\delta,j}(t^{-1}}\cdot)\subseteq[-2^{j-k+2},2^{j-k+2}]$. By finite speed of propagation \eqref{FS}, $K_{\phi_{\delta,j}(t^{-1}\sqrt{H})}\subseteq \mathfrak{D}_{2^{j-k+2}}$. It follows that, for any $t\in[2^{k-1},2^{k+2}]$
 \begin{align}\label{g1}
\Big|[b,\phi_{\delta,j}(t^{-1}\sqrt{H})]g\Big|
&=\Big|\sum_{\mathbf{m},\mathbf{m}'\in \mathbb{Z}^n:\mathrm{dist}(Q_\mathbf{m},Q_{\mathbf{m}'})\leq \sqrt{n}2^{j-k+3}}
\chi_{Q_\mathbf{m}}[b,\phi_{\delta,j}(t^{-1}\sqrt{H})]\chi_{Q_\mathbf{m'}}g\Big|\nonumber\\[4pt]
&\leq \sum_{\mathbf{m}\in \mathbb{Z}^n}
\Big|\chi_{Q_\mathbf{m}}[b,\phi_{\delta,j}(t^{-1}\sqrt{H})]\chi_{\widetilde{Q}_\mathbf{m}}g\Big|,
 \end{align}
where as usual $\chi_{Q_\mathbf{m}}$ is the characteristic function of $Q_\mathbf{m}$.

\hspace{-0.5cm} Let $b_{\widetilde{Q}_\mathbf{m}}=|\widetilde{Q}_\mathbf{m}|^{-1}\int_{\widetilde{Q}_\mathbf{m}}b(y)\mathrm{d}y$. Obviously,
 \begin{align}\label{g2}
 [b,\phi_{\delta,j}(t^{-1}\sqrt{H})]g
=(b-b_{\widetilde{Q}_\mathbf{m}})\phi_{\delta,j}(t^{-1}\sqrt{H})g-\phi_{\delta,j}(t^{-1}\sqrt{H})(b-b_{\widetilde{Q}_\mathbf{m}})g.
 \end{align}
It follows from \eqref{g1}, \eqref{g2}  and the disjointness of ${Q}_\mathbf{m}$ that
\begin{align*}
&\int_{2^{k-1}}^{2^{k+2}}\int_{\mathbb{R}^n}
   |[b,\phi_{\delta,j}(t^{-1}\sqrt{H})]\varphi_k(\sqrt{H})f(x)|^2(1+|x|)^{-\alpha}\mathrm{d}x\frac{\mathrm{d}t}{t}\\
 &=\sum_{\mathbf{m}\in \mathbb{Z}^n}\int_{2^{k-1}}^{2^{k+2}}\int_{\mathbb{R}^n}
   |\chi_{Q_\mathbf{m}}[b,\phi_{\delta,j}(t^{-1}\sqrt{H})]
   \chi_{\widetilde{Q}_\mathbf{m}}\varphi_k(\sqrt{H})f(x)|^2(1+|x|)^{-\alpha}\mathrm{d}x\frac{\mathrm{d}t}{t}\\
 &\leq C\sum_{\mathbf{m}\in \mathbb{Z}^n}\Big(E^{1}_{j,k,\mathbf{m}}+E^{2}_{j,k,\mathbf{m}}\Big),
\end{align*}
where
 \begin{align*}
E^{1}_{j,k,\mathbf{m}}&:=\int_{2^{k-1}}^{2^{k+2}}\int_{\mathbb{R}^n}
   |\chi_{Q_\mathbf{m}}(b-b_{\widetilde{Q}_\mathbf{m}})\phi_{\delta,j}(t^{-1}\sqrt{H})\chi_{\widetilde{Q}_\mathbf{m}}
   \varphi_k(\sqrt{H})f(x)|^2(1+|x|)^{-\alpha}\mathrm{d}x\frac{\mathrm{d}t}{t},\\
E^{2}_{j,k,\mathbf{m}}&:=\int_{2^{k-1}}^{2^{k+2}}\int_{\mathbb{R}^n}
   |\chi_{Q_\mathbf{m}}\phi_{\delta,j}(t^{-1}\sqrt{H})(b-b_{\widetilde{Q}_\mathbf{m}})\chi_{\widetilde{Q}_\mathbf{m}}
   \varphi_k(\sqrt{H})f(x)|^2(1+|x|)^{-\alpha}\mathrm{d}x\frac{\mathrm{d}t}{t}.
 \end{align*}
We select $r$ and $q$ such that $1/2=1/q+1/r$ and $\alpha r/2<n$ so that we can use the Lemma~\ref{L8}. By H\"older's inequality, we see
\begin{align}\label{3a1}
E^{1}_{j,k,\mathbf{m}}&\leq \|b-b_{\widetilde{Q}_\mathbf{m}}\|_{L^{q}(\widetilde{Q}_\mathbf{m})}^2
      \left(\int_{\mathbb{R}^n}|T_{j,k}^{\delta}(\chi_{\widetilde{Q}_\mathbf{m}}\varphi_k(\sqrt{H})f)(x)|^r(1+|x|)^{-\alpha r/2}\mathrm{d}x\right)^{2/r}.
\end{align}
By Lemma~\ref{L8} and H\"older's inequality again, we obtain that
\begin{align}\label{3a2}
 &\left(\int_{\mathbb{R}^n}|T_{j,k}^{\delta}(\chi_{\widetilde{Q}_\mathbf{m}}\varphi_k(\sqrt{H})f)(x)|^r(1+|x|)^{-\alpha r/2}\mathrm{d}x\right)^{2/r} \nonumber\\
 & \leq C_{\varepsilon,N} 2^{2(j_0-j)N} A_n^\varepsilon(\delta,r\alpha/2)^{2/r}2^{2kn(1-\frac2r)}
 \left(\int_{\widetilde{Q}_\mathbf{m}}|\varphi_k(\sqrt{H})f(x)|^{r'}(1+|x|)^{-\alpha r'/2}\mathrm{d}x\right)^{2/{r'}}\nonumber\\
 &\leq  C_{\varepsilon,N}2^{2(j_0-j)N} A_n^\varepsilon(\delta,r\alpha/2)^{2/r}2^{2kn(1-\frac2r)}
   |\widetilde{Q}_\mathbf{m}|^{\frac2q}   \int_{\widetilde{Q}_\mathbf{m}}|\varphi_k(\sqrt{H})f(x)|^{2}(1+|x|)^{-\alpha}\mathrm{d}x.
\end{align}
By John-Nirenberg's inequality, we see
\begin{align}\label{3a3}
\|b-b_{\widetilde{Q}_\mathbf{m}}\|_{L^{q}(\widetilde{Q}_\mathbf{m})}^2  \leq C_q\|b\|_{\mathrm{BMO}}^2|\widetilde{Q}_\mathbf{m}|^{\frac2q}.
\end{align}
Recall that $2^{j_0}\approx \delta^{-1}$ and $\widetilde{Q}_\mathbf{m}$ contained in  a cube with side length $5\cdot2^{j-k+2}$. In combination with estimates~\eqref{3a1}, \eqref{3a2} and \eqref{3a3} and $2/q=1-2/r$, we obtain
\begin{align*}
E^{1}_{j,k,\mathbf{m}}&\leq C_{\varepsilon,r,N}\|b\|_{\mathrm{BMO}}^22^{2(j_0-j)(N-n(1-\frac2r))}  \delta^{-2n(1-\frac2r)}A_n^\varepsilon(\delta,r\alpha/2)^{\frac2r}
      \int_{\widetilde{Q}_\mathbf{m}}|\varphi_k(\sqrt{H})f(x)|^{2}(1+|x|)^{-\alpha}\mathrm{d}x.
\end{align*}
  Similarly, by H\"older's inequality, Lemma~\ref{L8} and John-Nirenberg's inequality, we have
\begin{align*}
E^{2}_{j,k,\mathbf{m}}&\leq C_{\varepsilon,r,N}\|b\|_{\mathrm{BMO}}^22^{2(j_0-j)(N-n(1-\frac2r))}  \delta^{-2n(1-\frac2r)}A_n^\varepsilon(\delta,r\alpha/2)^{\frac2r}
      \int_{\widetilde{Q}_\mathbf{m}}|\varphi_k(\sqrt{H})f(x)|^{2}(1+|x|)^{-\alpha}\mathrm{d}x.
\end{align*}
  {\bf{Case 1}}. $n\geq2$ and $1<\alpha<n$.

For any  $0<\upsilon\leq1/2$ and $1<\alpha<n$, we choose $r$ such that
$$2<r<\min\Bigg\{\frac{2n}{\alpha},\frac{2(2n+3/2)}{2n+3/2-\upsilon}\Bigg\} .$$
See \eqref{AAA} for the definition of $A_n^\varepsilon(\delta,r\alpha/2)$, by calculation, we have
$$
A_n^\varepsilon(\delta,r\alpha/2)^{2/r}\delta^{-2n(1-2/r)}=\delta^{(3-\alpha)/2-(1-2/r)(2n+3/2)}
  < \delta^{(3-\alpha)/2-\upsilon},
  $$
then there exists a constant $C_1=C_{\alpha,\upsilon,N}$ such that
\begin{align}\label{f2}
E^{1}_{j,k,\mathbf{m}}+E^{2}_{j,k,\mathbf{m}}&\leq C_1\|b\|_{\mathrm{BMO}}^2 2^{2(j_0-j)(N-n(1-\frac2r))} \delta^{\frac{3-\alpha}{2}-\upsilon}\int_{\widetilde{Q}_\mathbf{m}}|\varphi_k(\sqrt{H})f(x)|^{2}(1+|x|)^{-\alpha}\mathrm{d}x.
\end{align}
{\bf{Case 2}}. $n=1$ and $0<\alpha<1$.

For any small $0<\varepsilon\leq1/4$ and $0<\upsilon'\leq1/4$, we choose $r$ such that $2<r<\min\{2/\alpha,2(3-\varepsilon)/(3-\varepsilon-\upsilon')\}.$ By  calculation,
$$
  A_1^\varepsilon(\delta,r\alpha/2)^{2/r}\delta^{-2(1-2/r)}=\delta^{1-\varepsilon-(1-2/r)(3-\varepsilon)}
  < \delta^{1-\varepsilon-\upsilon'}\leq \delta^{1-\upsilon},
$$
where $0<\upsilon\leq1/2$. Then there exists a constant $C_2=C_{\alpha,\upsilon,N}$ such that
\begin{align}\label{f3}
E^{1}_{j,k,\mathbf{m}}+E^{2}_{j,k,\mathbf{m}}
&\leq C_2\|b\|_{\mathrm{BMO}}^2 2^{2(j_0-j)(N-n(1-\frac2r))} \delta^{1-\upsilon}     \int_{\widetilde{Q}_\mathbf{m}}|\varphi_k(\sqrt{H})f(x)|^{2}(1+|x|)^{-\alpha}\mathrm{d}x.
\end{align}

\hspace{-0.5cm} {\bf{Case 3}}. $n\geq1$ and $\alpha=0$.

For any small $0<\upsilon\leq1/2$, we just need to select the $r$ such that $2<r<2(2n+1)/(2n+1-\upsilon)$. Then there exists a constant $C_3=C_{0,\upsilon,N}$ such that
\begin{align}\label{f4}
E^{1}_{j,k,\mathbf{m}}+E^{2}_{j,k,\mathbf{m}}\leq C_{3}\|b\|_{\mathrm{BMO}}^2 2^{2(j_0-j)(N-n(1-\frac2r))} \delta^{1-\upsilon}
 \int_{\widetilde{Q}_\mathbf{m}}|\varphi_k(\sqrt{H})f(x)|^{2}\mathrm{d}x.
\end{align}
Next we sum up the terms $E^{1}_{j,k,\mathbf{m}}+E^{2}_{j,k,\mathbf{m}}$ over $j$ and $\mathbf{m}$. Combining estimates~\eqref{f2}, \eqref{f3} and \eqref{f4},  choosing $N$ such that $N>n(1-2/r)$ and using the fact that $\RR^n=\cup_{\mathbf{m}\in \ZZ^n}Q_\mathbf{m}$  and $\{\widetilde{Q}_\mathbf{m}\}$ has finite overlaps, then we obtain that there exists a constant $C_{\alpha,\upsilon}$ such that
 \begin{align}\label{3a4}
 I_2(k) &\leq C_4\left(\sum_{j\geq j_0}2^{(j_0-j)(N-n(1-2/r))} \right)^2B^{\upsilon}_{\alpha,n}(\delta)\|b\|_{\mathrm{BMO}}^2 \sum_{\mathbf{m}\in \mathbb{Z}^n}
    \int_{\widetilde{Q}_\mathbf{m}}|\varphi_k(\sqrt{H})f(x)|^{2}(1+|x|)^{-\alpha}\mathrm{d}x\nonumber\\
 &\leq  C_{\alpha,\upsilon}B^{\upsilon}_{\alpha,n}(\delta)\|b\|_{\mathrm{BMO}}^2
    \int_{\mathbb{R}^n}|\varphi_k(\sqrt{H})f(x)|^{2}(1+|x|)^{-\alpha}\mathrm{d}x,
 \end{align}
where $B^{\upsilon}_{\alpha,n}(\delta)$ is defined in \eqref{BBB} and $C_4=C_1+C_2+C_3$.

\hspace{-0.4cm}Finally, summing up the terms $I_2(k)$ over $k$ in estimate~\eqref{3a4} and using the estimate~\eqref{ee5} in Lemma~\ref{L6}, we see
 \begin{align}\label{f5}
 \sum_{k\geq0}I_2(k)
 &\leq C_{\alpha,\upsilon}B^{\upsilon}_{\alpha,n}(\delta)\|b\|_{\mathrm{BMO}}^2
  \sum_{k\geq0}\int_{\mathbb{R}^n}|\varphi_k(\sqrt{H})f(x)|^{2}(1+|x|)^{-\alpha}\mathrm{d}x\nonumber\\
 &\leq C_{\alpha,\upsilon}B^{\upsilon}_{\alpha,n}(\delta)\|b\|_{\mathrm{BMO}}^2\int_{\mathbb{R}^n}|f(x)|^{2}(1+|x|)^{-\alpha}\mathrm{d}x,
\end{align}
where   $n=1,0\leq\alpha<1$;   $n\geq2, 1< \alpha<n$ or $\alpha=0$.

\hspace{-0.4cm}Combining estimates~\eqref{f0}, \eqref{f1} and \eqref{f5} and noting
$A^{\upsilon}_{n}(\delta,\alpha)<B^{\upsilon}_{\alpha,n}(\delta)$, we obtain that for any $0<\upsilon\leq1/2$,
\begin{align*}
\int_{0}^{\infty}\int_{\mathbb{R}^n}|[b,\phi_{\delta}(t^{-1}\sqrt{H})]f|^2(1+|x|)^{-\alpha}\mathrm{d}x\frac{\mathrm{d}t}{t}
\leq C_{\alpha,\upsilon}
B^{\upsilon}_{\alpha,n}(\delta)\|b\|_{\mathrm{BMO}}^2 \int_{\mathbb{R}^n}|f(x)|^2(1+|x|)^{-\alpha}\mathrm{d}x.
\end{align*}
Hence, we obtain Proposition~\ref{Pro1} provided Lemmas~\ref{L7} and \ref{L8} are proved.
\end{proof}

Finally, let us prove Lemmas~\ref{L7} and \ref{L8}.
\begin{proof}[Proof of Lemma~\ref{L7}]
We discuss this lemma by distinguish two cases: $\alpha=0$ and $0<\alpha<n$.

\hspace{-0.4cm}{\bf {Case 1. $\alpha=0$}}.

Pick up a function $\psi\in C_c^{\infty}(\mathbb{R})$ with support $\{s:1\leq |s|\leq4\}$  such that
$\sum_{\ell\in\mathbb{Z}}\psi(2^{-\ell}s)=1$ for $ s>0.$ For any $\delta>0$, let $\psi_{\delta,\ell}(s)=\psi(2^{-\ell}\delta^{-1}(1-s))$ for all $\ell\geq1$ and $\psi_{\delta,0}(s)=\psi_0(\delta^{-1}(1-s))$, where $\psi_0(s)=1-\sum_{\ell\geq1}\psi(2^{-\ell}s)$. Then for $k\geq 0$ and $j\geq j_0$,
\begin{align}\label{a2}
&\int_{2^{k-1}}^{2^{k+2}}\int_{\mathbb{R}^n}|\phi_{\delta,j}(t^{-1}\sqrt{H})f(x)|^2\mathrm{d}x\frac{\mathrm{d}t}{t}\leq \left(\sum_{\ell\geq0}P_{k,j,\ell}(f) ^{1/2}\right)^2,
\end{align}
where
$$P_{k,j,\ell}(f)=\int_{2^{k-1}}^{2^{k+2}}\int_{\mathbb{R}^n}|(\phi_{\delta,j}\psi_{\delta,\ell})(t^{-1}\sqrt{H})f(x)|^2
\mathrm{d}x\frac{\mathrm{d}t}{t}.$$
Let $i=0,1,\cdots,i_0=[8\delta^{-1}]+1$, $I_i$ is defined by
$$I_i=[2^{k-1}+i2^{k-1}\delta,2^{k-1}+(i+1)2^{k-1}\delta].$$
We decompose $[2^{k-1},2^{k+2}]$ into some intervals $\{I_i\}$ with $[2^{k-1},2^{k+2}]\subseteq \bigcup_{i=0}^{i_0}I_i$. We see that the $\mathrm{d}t/t$ measure of the $I_i$ is less than $\delta$. We also define a function $\zeta_i$  associated with  $I_i$ by
 $$\zeta_i(s)=\zeta\left(i+\frac{2^{k-1}-s}{2^{k-1}\delta}\right),$$
 where $\zeta\in C_c^{\infty}([-1,1])$ and $\sum_{i\in\ZZ}\zeta(\cdot-i)=1$.
 Since  $\psi_{\delta,\ell}(t^{-1}s)\zeta_{i'}(s)= 0$ for every $t\in I_i$ if $i'\notin [i-2^{\ell+6}, i+2^{\ell+6}]$ , then for any $t\in I_i$,
 \begin{align}\label{a3}
 |(\phi_{\delta,j}\psi_{\delta,\ell})(t^{-1}\sqrt{H})f|^2
 &=\Big|\sum_{i'=i-2^{\ell+6}}^{i+2^{\ell+6}}(\phi_{\delta,j}\psi_{\delta,\ell})(t^{-1}\sqrt{H})\zeta_{i'}(\sqrt{H})f\Big|^2\nonumber\\
 &\leq C2^{\ell}\sum_{i'=i-2^{\ell+6}}^{i+2^{\ell+6}}\left|(\phi_{\delta,j}\psi_{\delta,\ell})(t^{-1}\sqrt{H})\zeta_{i'}(\sqrt{H})f\right|^2.
 \end{align}
Then
\begin{align}\label{a4}
P_{k,j,\ell}(f)\leq  C2^\ell\sum_{i=0}^{i_0}\sum_{i'=i-2^{\ell+6}}^{i+2^{\ell+6}}
\int_{I_i}\int_{\mathbb{R}^n}|(\phi_{\delta,j}\psi_{\delta,\ell})(t^{-1}\sqrt{H})\zeta_{i'}(\sqrt{H})f(x)|^2\mathrm{d}x\frac{\mathrm{d}t}{t}.
\end{align}
From estimate~\eqref{d1} and supp\,$\psi_{\delta,\ell}\subseteq [1-2^{\ell+2}\delta,1-2^{\ell}\delta]\cup[1+2^{\ell}\delta,1+2^{\ell+2}\delta]$ , the function $\phi_{\delta,j}\psi_{\delta,\ell}$ satisfies
\begin{align}\label{d2}
  \|(\phi_{\delta,j}\psi_{\delta,\ell})\|_{L^{\infty}(\RR)}\leq C2^{(j_0-j)N}2^{-\ell N}\ \  \mathrm{for}\  \ell\geq0,j\geq j_0,
\end{align}
which, together with  the $L^2$-boundedness of the  spectral multiplier implies that
 \begin{align}\label{a6}
 \|(\phi_{\delta,j}\psi_{\delta,\ell})(t^{-1}\sqrt{H})\zeta_{i'}(\sqrt{H})f\|_{L^2(\RR^n)}&\leq \|\phi_{\delta,j}\psi_{\delta,\ell}\|_{L^{\infty}(\RR)}\|\zeta_{i'}(\sqrt{H})f\|_{L^2(\RR^n)}\nonumber\\
 &\leq C_N2^{-\ell N}2^{(j_0-j)N}\|\zeta_{i'}(\sqrt{H})f\|_{L^2(\RR^n)}.
 \end{align}
 Taking estimate~\eqref{a6} into estimate~\eqref{a4} and using Minkowski's inequality give that
 \begin{align}\label{3b1}
P_{k,j,\ell}(f)
&\leq  C_N2^{-(2N-1)\ell}2^{2(j_0-j)N}\sum_{i=0}^{i_0}
\sum_{i'=i-2^{\ell+6}}^{i+2^{\ell+6}}\int_{I_i}\|\zeta_{i'}(\sqrt{H})f\|_{L^2(\RR^n)}^2\frac{\mathrm{d}t}{t}\nonumber\\
&= C_N2^{-(2N-1)\ell}2^{2(j_0-j)N}\sum_{i'=-2^{\ell+6}}^{i_0+2^{\ell+6}}
\sum_{\{i\in\NN\cap[0, i_0],\ |i-i'|\leq2^{\ell+6}\}}
\int_{I_i}\|\zeta_{i'}(\sqrt{H})f\|_{L^2(\RR^n)}^2\frac{\mathrm{d}t}{t}.
 \end{align}
The decomposition of $[2^{k-2},2^{k+1}]$ into $\{I_i\}_{i=0}^{i_0}$ makes that the $\mathrm{d}t/t$ measure of the interval $I_i$ is uniformly  less than $\delta$, indeed,
\begin{align*}
  \int_{2^{k-1}(1+i\delta)}^{2^{k-1}(1+(i+1)\delta)}1\ \frac{\mathrm{d}t}{t}=\ln\left(1+\frac{\delta}{1+i\delta}\right)\leq \ln(1+\delta)  \leq \delta.
\end{align*}
It yields that
\begin{align}\label{3b2}
\sum_{\{i\in\NN\cap[0, i_0],\ |i-i'|\leq2^{\ell+6}\}}
\int_{I_i}\|\zeta_{i'}(\sqrt{H})f\|_{L^2(\RR^n)}^2\frac{\mathrm{d}t}{t}
&\leq \delta\sum_{\{i\in\NN\cap[0, i_0],\ |i-i'|\leq2^{\ell+6}\}}
\|\zeta_{i'}(\sqrt{H})f\|_{L^2(\RR^n)}^2\nonumber\\
&\leq C2^{\ell}\delta
\|\zeta_{i'}(\sqrt{H})f\|_{L^2(\RR^n)}^2.
\end{align}
Combining estimates~\eqref{3b1}, \eqref{3b2} and using $L^2$-estimate of square function for Hermite operator give that
 \begin{align}\label{a5}
P_{k,j,\ell}(f)
&\leq C_N2^{-(2N-2)\ell}2^{2(j_0-j)N}\delta\sum_{i'\in\ZZ}
\|\zeta_{i'}(\sqrt{H})f\|_{L^2(\RR^n)}^2\nonumber\\
&\leq C_N2^{-(2N-2)\ell}2^{2(j_0-j)N}\delta\|f\|_{L^2(\RR^n)}^2.
\end{align}
Combining estimates~\eqref{a2} and \eqref{a5} and summing up the terms $P_{k,j,\ell}(f)$ over $\ell$ give that
\begin{align*}
&\int_{2^{k-1}}^{2^{k+2}} \int_{\mathbb{R}^n}|\phi_{\delta,j}(t^{-1}\sqrt{H})f(x)|^2\mathrm{d}x\frac{\mathrm{d}t}{t}\leq C_N2^{2(j_0-j)N}\delta\|f\|_{L^2(\RR^n)}^2.
\end{align*}

\hspace{-0.5cm}{\bf {Case 2. $0<\alpha<n$}}.
 We discuss the estimate into two cases: $2^{k}\geq2\delta^{-1/2}$ and $1\leq2^{k}<2\delta^{-1/2}$.

If $k\geq-\log_2\delta^{1/2}+1$,  we use \cite[(3.38) and (3.39)]{CDHLY} with $f$ in place of $\varphi_k(\sqrt{H})f$ to obtain
\begin{align}\label{ll0}
&\int_{2^{k-1}}^{2^{k+2}} \int_{\mathbb{R}^n}|\phi_{\delta,j}(t^{-1}\sqrt{H})f(x)|^2(1+|x|)^{-\alpha}\mathrm{d}x\frac{\mathrm{d}t}{t}\leq C_N2^{2(j_0-j)N}A^{\varepsilon}_{n}(\delta,\alpha)\int_{\mathbb{R}^n}|f(x)|^2(1+|x|)^{-\alpha}\mathrm{d}x.
\end{align}
The procedure is still valid if $\varphi_k(\sqrt{H})f$ in \cite[(3.38) and (3.39)]{CDHLY} is replaced  by $f$. This procedure rested on the trace lemma, the weighted Plancherel-type estimate~\eqref{e4} and the localization strategy which is based on finite speed of propagation. Indeed, estimate~\eqref{e4} gives the bound of spectral multipliers from $L^2(\RR^n)$ to $L^2(\RR^n,(1+|x|)^{-\alpha})$. To come back to $L^2(\RR^n,(1+|x|)^{-\alpha})$, when the physical space is  near $\{x\in\RR^n:|x|\leq 2^{j-k+2}\}$, it added the weight $(1+|x|)^{-\alpha}$ into the integral, which would bring a factor $2^{(j-k+1)\alpha}$ for $j>k$.  This factor is advantageous when  $2^{k}\geq2\delta^{-1/2}$, however it's helpless when   $1\leq2^{k}<2\delta^{-1/2}$. To overcome this, based on the Lemma~\ref
{L2}, we would instead consider the operator:
$\int_{2^{k-1}}^{2^{k+2}} |\phi_{\delta,j}(t^{-1}\sqrt{H})(I+H)^{\alpha/4}
f(x)|^2\mathrm{d}t/t$
and estimate its bound from $L^2(\RR^n)$ to $L^2(\RR^n,(1+|x|)^{-\alpha})$.

 If $0\leq k<-\log_2\delta^{1/2}+1$,  we claim that for any $0<\varepsilon\leq1/2$, there exists a constant $C_{\varepsilon,N}$ such that
\begin{align}\label{ll1}
&\int_{2^{k-1}}^{2^{k+2}} \int_{\mathbb{R}^n}|\phi_{\delta,j}(t^{-1}\sqrt{H})(I+H)^{\alpha/4}
f(x)|^2(1+|x|)^{-\alpha}\mathrm{d}x\frac{\mathrm{d}t}{t}\leq C_{\varepsilon,N}2^{2(j_0-j)N}A^{\varepsilon}_{n}(\delta,\alpha)\int_{\mathbb{R}^n}|f(x)|^2\mathrm{d}x.
\end{align}
Using estimate~\eqref{ll1} and  Lemma~\ref{L2}, we have
\begin{align*}
\int_{2^{k-1}}^{2^{k+2}} \int_{\mathbb{R}^n}|\phi_{\delta,j}(t^{-1}\sqrt{H})
f(x)|^2(1+|x|)^{-\alpha}\mathrm{d}x\frac{\mathrm{d}t}{t}&\leq C_{\varepsilon,N}2^{2(j_0-j)N}A^{\varepsilon}_{n}(\delta,\alpha)\int_{\mathbb{R}^n}|(I+H)^{-\alpha/4}f(x)|^2\mathrm{d}x\\
&\leq C_{\varepsilon,N}2^{2(j_0-j)N}A^{\varepsilon}_{n}(\delta,\alpha)\int_{\mathbb{R}^n}|f(x)|^2(1+|x|)^{-\alpha}\mathrm{d}x,
\end{align*}
which in combination with estimate~\eqref{ll0} proves Lemma~\ref{L7}.

We now turn to verify the claim. Similar to the discussion of  estimate~\eqref{a2}, we see
\begin{align}\label{b1}
&\int_{2^{k-1}}^{2^{k+2}} \int_{\mathbb{R}^n}|\phi_{\delta,j}(t^{-1}\sqrt{H})(I+H)^{\alpha/4}f(x)|^2
(1+|x|)^{-\alpha}\mathrm{d}x\frac{\mathrm{d}t}{t}\leq \left(\sum_{\ell\geq0}  Q_{j,k,\ell}(f)^{1/2} \right)^2,
\end{align}
where
$$Q_{j,k,\ell}(f)=\int_{2^{k-1}}^{2^{k+2}} \int_{\mathbb{R}^n}|(\phi_{\delta,j}\psi_{\delta,\ell})(t^{-1}\sqrt{H})(I+H)^{\alpha/4}f(x)|^2
(1+|x|)^{-\alpha}\mathrm{d}x\frac{\mathrm{d}t}{t}.$$
Arguing as estimates~\eqref{a3} and \eqref{a4}, we have
\begin{align*}
Q_{j,k,\ell}(f)\leq C2^{\ell}\sum_{i=0}^{i_0}\sum_{i'=i-2^{\ell+6}}^{i+2^{\ell+6}}
\int_{I_i}\int_{\mathbb{R}^n}|(\phi_{\delta,j}\psi_{\delta,\ell})(t^{-1}\sqrt{H})(I+H)^{\alpha/4}
\zeta_{i'}(\sqrt{H})f(x)|^2(1+|x|)^{-\alpha}\mathrm{d}x\frac{\mathrm{d}t}{t}.
\end{align*}
The function $\psi_{\delta,\ell}(t^{-1}s)$ is support in $[t(1-2^{\ell+2}\delta),t(1+2^{\ell+2}\delta)]$. Let $R=1+[t(1+2^{\ell+2}\delta)]$. When $n\geq2$, using  estimate~\eqref{e4} where $1<\alpha<n$, then we have
\begin{align}\label{b2}
&\int_{\mathbb{R}^n}|(\phi_{\delta,j}\psi_{\delta,\ell})(t^{-1}\sqrt{H})(I+H)^{\alpha/4}\zeta_{i'}(\sqrt{H})f(x)|^2(1+|x|)^{-\alpha}\mathrm{d}x\nonumber\\
&\leq R\|(\phi_{\delta,j}\psi_{\delta,\ell})(t^{-1}Rs)(I+|Rs|^2)^{\alpha/4}\|_{R^2,_2}^2\|\zeta_{i'}(\sqrt{H})f\|_{L^2(\RR^n)}^2.
\end{align}
We see that
\begin{align}\label{3c1}
&\|(\phi_{\delta,j}\psi_{\delta,\ell})(t^{-1}Rs)(I+|Rs|^2)^{\alpha/4}\|_{R^2,_2}^2\nonumber\\
&\leq \|(\phi_{\delta,j}\psi_{\delta,\ell})(t^{-1}Rs)(I+|Rs|^2)^{\alpha/4}\|_{L^\infty(\RR)}^2
\|\chi_{[tR^{-1}(1-2^{\ell+2}\delta),tR^{-1}(1+2^{\ell+2}\delta)]}\|_{R^2,_2}^2.
\end{align}
From the estimate~\eqref{d2} and the support property of $\psi_{\delta,\ell}$, we know that
\begin{align}\label{m1}
\|(\phi_{\delta,j}\psi_{\delta,\ell})(t^{-1}Rs)(1+|Rs|^2)^{\alpha/4}\|_{L^\infty(\RR)}^2\leq C_N2^{-2\ell N}2^{2(j_0-j)N}R^{\alpha}.
\end{align}
The length of the interval $[tR^{-1}(1-2^{\ell+2}\delta),tR^{-1}(1+2^{\ell+2}\delta)]$ is $8tR^{-1}2^\ell\delta$, thus  $$\#\Big\{i\in \ZZ:\Big[\frac{i-1}{R^2},\frac{i}{R^2}\Big]\cap\left[tR^{-1}(1-2^{\ell+2}\delta),tR^{-1}(1+2^{\ell+2}\delta)\right]\neq \emptyset\Big\}\leq 1+8Rt2^\ell\delta,$$
which implies that
\begin{align}\label{m2}
\big\|\chi_{[tR^{-1}(1-2^{\ell+2}\delta),tR^{-1}(1+2^{\ell+2}\delta)]}\big\|_{R^2,_2}^2
&=\frac{1}{R^2}\sum_{i=-R^2+1}^{R^2}\sup_{s\in[\frac{i-1}{R^2},\frac{i}{R^2})}
\big|\chi_{[tR^{-1}(1-2^{\ell+2}\delta),tR^{-1}(1+2^{\ell+2}\delta)]}(s)\big|^2\nonumber\\
&\leq R^{-2}\min\{R^2,(1+8Rt2^\ell\delta)\}.
\end{align}
Combining estimates~\eqref{3c1}-\eqref{m2} gives that
\begin{align*}
{\rm{RHS} \ of \  \eqref{b2}}
&\leq C_N2^{-2\ell N}2^{2(j_0-j)N}R^{\alpha-1}\min\{R^2,(1+Rt2^\ell\delta)\}\|\zeta_{i'}(\sqrt{H})f\|_{L^2(\RR^n)}^2.
\end{align*}
If $2^\ell\delta\leq1$, note that $1/2\leq t\approx2^k\leq C\delta^{-1/2}$, it's easy to see that $R\approx t\approx 2^k\leq C\delta^{-1/2}$, then
$$R^{\alpha-1}\min\{R^2,(1+Rt2^\ell\delta)\}\leq R^{\alpha-1}(1+R^22^{\ell}\delta)
\leq C2^{k(\alpha-1)}(1+2^{\ell+2k}\delta)
\leq C2^{\ell}\delta^{(1-\alpha)/2}.$$
If $2^\ell\delta\geq1$,  note that $1/2\leq t\approx 2^k\leq C\delta^{-1/2}$, it's easy to see that $R\leq Ct2^\ell\delta$, then
$$R^{\alpha-1}\min\{R^2,(1+Rt2^\ell\delta)\}\leq CR^{\alpha+1}\leq C(t2^\ell\delta)^{\alpha+1}
\leq C(2^{\ell+k}\delta)^{\alpha+1}
\leq 2^{(\alpha+1)\ell}\delta^{(\alpha+1)/2}.$$
Therefore, if $2^k\leq C\delta^{-1/2}$,
\begin{align}\label{m3}
R^{\alpha-1}\min\{R^2,(1+Rt2^\ell\delta)\}\leq C2^{(\alpha+1)\ell}\delta^{(1-\alpha)/2}.
\end{align}
As a consequence, for any $N>(\alpha+3)/2$
\begin{align*}
{\rm{RHS} \ of \  \eqref{b2}}&\leq C_N2^{-\ell(2N-\alpha-1)}2^{2(j_0-j)N}\delta^{(1-\alpha)/2}\|\zeta_{i'}(\sqrt{H})f\|_{L^2(\RR^n)}^2,
\end{align*}
which in combination with the $\mathrm{d}t/t$ measure of $I_i$ is uniformly less than $\delta$, Fubini's theorem and the $L^2$-estimate of square function  for Hermite operator yields
\begin{align}\label{b3}
Q_{j,k,\ell}(f)&\leq C_N2^{-\ell(2N-\alpha-2)}2^{2(j_0-j)N} \delta^{(1-\alpha)/2}\sum_{i=0}^{i_0}\sum_{i'=i-2^{\ell+6}}^{i+2^{\ell+6}}
\int_{I_i}\|\zeta_{i'}(\sqrt{H})f\|_{L^2(\RR^n)}^2\frac{\mathrm{d}t}{t}\nonumber\\
&=C_N2^{-\ell(2N-\alpha-2)}2^{2(j_0-j)N} \delta^{(1-\alpha)/2}\sum_{i'=-2^{\ell+6}}^{i_0+2^{\ell+6}}\sum_{\{i\in\NN\cap[0, i_0],\ |i-i'|\leq2^{\ell+6}\}}
\int_{I_i}\|\zeta_{i'}(\sqrt{H})f\|_{L^2(\RR^n)}^2\frac{\mathrm{d}t}{t}\nonumber\\
&\leq C_N2^{-\ell(2N-\alpha-2)}2^{2(j_0-j)N} \delta^{(1-\alpha)/2}2^{\ell}\delta\sum_{i'\in \ZZ}
\|\zeta_{i'}(\sqrt{H})f\|_{L^2(\RR^n)}^2\nonumber\\
&\leq C_N2^{-\ell(2N-\alpha-3)}2^{2(j_0-j)N} \delta^{(3-\alpha)/2}\|f\|_{L^2(\RR^n)}^2.
\end{align}
Putting the estimate~\eqref{b3} into estimate~\eqref{b1} and  taking the sum over $\ell$ yield
\begin{align}\label{b4}
\int_{2^{k-1}}^{2^{k+2}} \int_{\mathbb{R}^n}|\phi_{\delta,j}(t^{-1}\sqrt{H})(I+H)^{\alpha/4}f(x)|^2
(1+|x|)^{-\alpha}\mathrm{d}x\frac{\mathrm{d}t}{t}\leq C_N2^{2(j_0-j)N}\delta^{(3-\alpha)/2}\|f\|_{L^2(\RR^n)}^2.
\end{align}
When $n=1$, this situation can be showed in the same manner as before, the difference is that we use the case  where $0<\alpha<1$ in estimate~\eqref{e4}. Similar to the discussion of \eqref{m2}, we have
\begin{align}\label{m5}
\big\|\chi_{[tR^{-1}(1-2^{\ell+2}\delta),tR^{-1}(1+2^{\ell+2}\delta)]}\big\|_{R^2,_{\frac{2(1+\varepsilon)}{\alpha}}}^2
\leq \left(R^{-2}\min\{R^2,(1+8Rt2^\ell\delta)\}\right)^{\frac{\alpha}{1+\varepsilon}}.
\end{align}
From  estimate~\eqref{e4} for the case $n=1,\ 0<\alpha<1$ and estimates~\eqref{m1} and \eqref{m5} , then for any  $\varepsilon>0$, there exist  constants $C_{\varepsilon}$ and $C_{\varepsilon,N}$  such that
\begin{align}\label{m4}
&\int_{\mathbb{R}^n}|(\phi_{\delta,j}\psi_{\delta,\ell})(t^{-1}\sqrt{H})(I+H)^{\alpha/4}\zeta_{i'}(\sqrt{H})f(x)|^2(1+|x|)^{-\alpha}\mathrm{d}x\nonumber\\
&\leq C_{\varepsilon} R^{\frac{\alpha}{1+\varepsilon}}\|(\phi_{\delta,j}\psi_{\delta,\ell})(t^{-1}Rs)(I+|Rs|^2)^{\alpha/4}
\|_{R^2,_{\frac{2(1+\varepsilon)}{\alpha}}}^2\|\zeta_{i'}(\sqrt{H})f\|_{L^2(\RR^n)}^2\nonumber\\
&\leq C_{\varepsilon,N}2^{-2\ell N}2^{2(j_0-j)N}
\left(R^{\varepsilon}\min\{R^2,(1+Rt2^\ell\delta)\}\right)^{\frac{\alpha}{1+\varepsilon}}
\|\zeta_{i'}(\sqrt{H})f\|_{L^2(\RR^n)}^2.
\end{align}
Similarly to the discussion of \eqref{m3}, if $2^k\leq C\delta^{-1/2}$, we have
\begin{align}\label{m6}
  R^{\varepsilon}\min\{R^2,(1+Rt2^\ell\delta)\}\leq 2^{(2+\varepsilon)\ell}\delta^{-\varepsilon/2}.
\end{align}
Putting estimate~\eqref{m6} into estimate~\eqref{m4} yields
\begin{align*}
{\rm{RHS} \ of \ \eqref{m4}}&\leq C_{\varepsilon,N}2^{-\ell \big(2N-\alpha\frac{2+\varepsilon}{1+\varepsilon}\big)}2^{2(j_0-j)N}\delta^{-\frac{\alpha\varepsilon}{2(1+\varepsilon)}}\|\zeta_{i'}(\sqrt{H})f\|_{L^2(\RR^n)}^2\\
&\leq C_{\varepsilon,N}2^{-2\ell (N-\alpha)}2^{2(j_0-j)N}\delta^{-\varepsilon}\|\zeta_{i'}(\sqrt{H})f\|_{L^2(\RR^n)}^2.
\end{align*}
In the same manner as estimates~\eqref{b3} and \eqref{b4}, we have
 $$Q_{j,k,\ell}(f)\leq C_{\varepsilon,N}2^{-\ell (2N-2\alpha-2)}2^{2(j_0-j)N}\delta^{1-\varepsilon}\|f\|_{L^2(\RR^n)}^2, $$
 and
 \begin{align*}
\int_{2^{k-1}}^{2^{k+2}} \int_{\mathbb{R}^n}|\phi_{\delta,j}(t^{-1}\sqrt{H})(I+H)^{\alpha/4}f(x)|^2
(1+|x|)^{-\alpha}\mathrm{d}x\frac{\mathrm{d}t}{t}\leq C_{\varepsilon,N}2^{2(j_0-j)N}\delta^{1-\varepsilon }\|f\|_{L^2(\RR^n)}^2,
\end{align*}
which in combination with estimate~\eqref{b4} yields \eqref{ll1}, so the claim holds. The proof of Lemma~\ref{L7} is complete.
\end{proof}

\begin{proof}[Proof of Lemma~\ref{L8}]
 The proof of the lemma is inspired by \cite{HL3}. For convenience, we  use $\|T\|_{p\to q} $ for the operator norm of $T$ if $T$ is a bounded linear operator from $L^p(\mathbb{R}^n)$ to $L^q(\mathbb{R}^n)$ for given $1\leq p,q\leq\infty$. We will firstly obtain that $T_{j,k}^{\delta}$ is bounded from $L^1(\mathbb{R}^n)$ to $L^{\infty}(\mathbb{R}^n)$. To do so, we need to estimate $\|\phi_{\delta,j}(t^{-1}\sqrt{H})\|_{1\rightarrow\infty}$.
\begin{align*}
  \|\phi_{\delta,j}(t^{-1}\sqrt{H})\|_{1\rightarrow\infty}&\leq \|\phi_{\delta,j}^{1/2}(t^{-1}\sqrt{H})\|_{1\rightarrow2}\|\phi_{\delta,j}^{1/2}(t^{-1}\sqrt{H})\|_{2\rightarrow\infty}\nonumber\\
  &\leq \|\phi_{\delta,j}^{1/2}(t^{-1}\sqrt{H})\|^2_{2\rightarrow\infty}\leq\left(\sum_{\ell\geq0}\|(\phi_{\delta,j}^{1/2}\psi_{\delta,\ell})(t^{-1}\sqrt{H})\|_{2\rightarrow\infty}\right)^2.
\end{align*}
By \cite[Lemma~2.2]{DOS} and  estimate~\eqref{d2}, we have
\begin{align*}
\|(\phi_{\delta,j}^{1/2}\psi_{\delta,\ell})(t^{-1}\sqrt{H})\|_{2\rightarrow\infty}&\leq \left(\int_{\mathbb{R}^n}|K_{(\phi^{1/2}_{\delta,j}\psi_{\delta,\ell})(t^{-1}\sqrt{H})}(x,y)|^2\mathrm{d}y\right)^{1/2}\\
&\leq t^{n/2}\|\phi_{\delta,j}^{1/2}\psi_{\delta,\ell}\|_{L^\infty(\RR)}\leq C_Nt^{n/2}2^{-\ell N/2}2^{(j_0-j)N/2}.
\end{align*}
Summing the terms $\|(\phi_{\delta,j}^{\frac12}\psi_{\delta,\ell})(t^{-1}\sqrt{H})\|_{2\rightarrow\infty}$ over $\ell$ gives that
$$\|\phi_{\delta,j}(t^{-1}\sqrt{H})\|_{1\rightarrow\infty}\leq C_N t^n2^{(j_0-j)N}.$$
Finally, for any $N\in\mathbb{N}$ and $j\geq j_0$,
 \begin{align}\label{b6}
    \|T_{j,k}^{\delta}(f)\|_{L^\infty(\RR^n)}
   & \leq  \left(\int_{2^{k-1}}^{2^{k+2}} ||\phi_{\delta,j}(t^{-1}\sqrt{H})f\|_{L^\infty(\RR^n)}^2\frac{\mathrm{d}t}{t}\right)^{1/2} \leq C_N2^{kn}2^{(j_0-j)N}\|f\|_{L^1(\RR^n)}.
  \end{align}
$T_{j,k}^{\delta}$ is a sub-linear operator. Using interpolation(see \cite[Theorem 2.11]{St3} or \cite[Page 120]{BL}) with estimate~\eqref{b6} and Lemma~\ref{L7}, we obtain that
 \begin{align*}
\|T_{j,k}^{\delta}(f)\|_{L^r(\mathbb{R}^n,(1+|x|)^{-\alpha})}\leq
C_{\varepsilon,N}2^{(j_0-j)N}2^{(1-\theta) kn}A^{\varepsilon}_{n}(\delta,\alpha)^{\theta/2}\|f\|_{L^{r'}(\mathbb{R}^n,(1+|x|)^{-\alpha(r'-1)})},
 \end{align*}
 where  $n=1,0\leq\alpha<1$;   $n\geq2, 1< \alpha<n$ or $\alpha=0$; $1/r=\theta/2$ and $2< r< \infty$.
 \end{proof}


\section{ Proof of Theorem~\ref{thm1.1}}

We now begin to prove Theorem~\ref{thm1.1}.
\begin{proof} Observe that for any $\lambda>0$ and  $\rho<\lambda$
  \begin{align*}
    \left(1-\frac{m^2}{R^2}\right)^{\lambda}_+ & =\widetilde{C}_{\lambda,\,\rho} R^{-2\lambda}\int_{0}^{R}(R^2-t^2)^{\lambda-\rho-1}t^{2\rho+1}\left(1-\frac{m^2}{t^2}\right)^\rho_+\mathrm{d}t.
  \end{align*}
For $\rho<\lambda-1/2$ and $\rho>-1/2$,   we apply  the H\"older inequality to obtain
\begin{align}\label{4a}
  |\,[b,S_R^{\lambda}(H)]f|&=\widetilde{C}_{\lambda,\,\rho}R^{-2\lambda}\big|\int_{0}^{R}(R^2-t^2)^{\lambda-\rho-1}t^{2\rho+1}
  [b,S_t^{\rho}(H)]f\mathrm{d}t\big|\nonumber\\
  &\leq C_{\lambda,\,\rho}\left(\frac1R\int_{0}^{R}|\,[b,S_t^{\rho}(H)]f|^2\mathrm{d}t\right)^{1/2}.
\end{align}
 Note that $x_+^{\rho}=\sum_{k\in\mathbb{Z}}2^{-k\rho}\phi(2^kx)$ for some $\phi\in C_c^{\infty} ([1/8,1/2])$. Let
 $\phi_k(s)=\phi(2^k(1-s^2))$, $k\geq1$ and $\phi_0(s)=\sum_{k\leq0}2^{-k\rho}\phi(2^k(1-s))$, for $s>0$. Then
 \begin{align}\label{4b}
[b,S_t^{\rho}(H)]f=[b,\phi_0(t^{-2}H)]f+\sum_{k\geq1}2^{-k\rho}[b,\phi_k(t^{-1}\sqrt{H})]f.
\end{align}
It follows from estimate~\eqref{4a} and equality~\eqref{4b} that
 \begin{align*}
   \sup_{R>0}|[b,S_R^{\lambda}(H)]f|  \leq&C_{\lambda,\,\rho}\left( \sup_{t>0}|[b,\phi_0(t^{-2}H)]f|
   +\sum_{k\geq1}2^{-k\rho}\left(\int_{0}^{\infty}|\,[b,\phi_k(t^{-1}\sqrt{H})]f|^2\frac{\mathrm{d}t}{t}\right)^{1/2}\right).
 \end{align*}
 Note that $\phi_0\in C_c^{\infty}$ with support $\{s:0\leq|s|\leq1\}$. By Lemma~\ref{L4}, for any $0\leq\alpha<n$
\begin{align}\label{4d}
\Big\|\sup_{t>0}\left|[b,\phi_0(t^{-2}H)]f\right|\ \Big\|_{L^2(\mathbb{R}^n,(1+|x|)^{-\alpha})}
 \leq C\|b\|_{\mathrm{BMO}}\|f\|_{L^2(\mathbb{R}^n,(1+|x|)^{-\alpha})}.
\end{align}
Let $\rho=\lambda-1/2-\eta$ for some $\eta>0$. If $\alpha=0$ and $n\in \NN^+$, by Proposition~\ref{Pro1}, for any $0<\upsilon\leq1/2$ there exists a constant $C_{0,\upsilon}>0$ such that
\begin{align}\label{4e}
\left\|\sum_{k\geq1}2^{-k\rho}\left(\int_{0}^{\infty}|[b,\phi_k(t^{-1}\sqrt{H})]f|^2
\frac{\mathrm{d}t}{t}\right)^{\frac12}\right\|_{L^2(\RR^n)}
&\leq \sum_{k\geq1}2^{-k\rho}\left\|\left(\int_{0}^{\infty}|[b,\phi_k(t^{-1}\sqrt{H})]f|^2
\frac{\mathrm{d}t}{t}\right)^{\frac12}\right\|_{L^2(\RR^n)}\nonumber\\
&\leq C_{0,\upsilon}\|b\|_{\mathrm{BMO}}\sum_{k\geq1}2^{-k(\lambda-\eta-\upsilon)}\|f\|_{L^2(\RR^n)}.
\end{align}
Then the $\mathrm{RHS}$ of \eqref{4e} is bounded by $C_{0,\upsilon}\|b\|_{\mathrm{BMO}}\left\|f\right\|_{L^2(\mathbb{R},(1+|x|)^{-\alpha})}$ provided  $\eta,\upsilon$ are small enough. From estimates~\eqref{4d} and \eqref{4e}, we can conclude that  $\sup_{R>0}|[b,S_R^{\lambda}(H)]f|$ is bounded on $L^2(\mathbb{R}^n)$ for any $n\in \mathbb{N}^+$.

Similarly, if $n=1$ and $0<\alpha<1$, by Proposition~\ref{Pro1}, for any $0<\upsilon\leq1/2$ there exists  a constant $C_{\alpha,\upsilon}>0$ such that
\begin{align}\label{4f}
&\left\|\sum_{k\geq1}2^{-k\rho}\left(\int_{0}^{\infty}|[b,\phi_k(t^{-1}\sqrt{H})]f|^2
\frac{\mathrm{d}t}{t}\right)^{1/2}\right\|_{L^2(\mathbb{R},(1+|x|)^{-\alpha})}\nonumber\\
&\leq C_{\alpha,\upsilon}\|b\|_{\mathrm{BMO}}\sum_{k\geq1}2^{-k(\lambda-\eta-\upsilon)}\left\|f\right\|_{L^2(\mathbb{R},(1+|x|)^{-\alpha})}.
\end{align}
Then  the $\mathrm{RHS}$ of \eqref{4f} is bounded by $C_{\alpha,\upsilon}\|b\|_{\mathrm{BMO}}\left\|f\right\|_{L^2(\mathbb{R},(1+|x|)^{-\alpha})}$ provided  $\eta,\upsilon$ are small enough. From estimates~\eqref{4d} and \eqref{4f}, we can conclude that $\sup_{R>0}|[b,S_R^{\lambda}(H)]f|$ is bounded on $L^2(\mathbb{R},(1+|x|)^{-\alpha})$ for $0< \alpha<1$.

If $n\geq2$ and $1<\alpha<n$, by Proposition~\ref{Pro1}, for any $0<\upsilon\leq1/2$ there exists a constant $C_{\alpha,\upsilon}>0$ such that
\begin{align}\label{4g}
&\left\|\sum_{k\geq1}2^{-k\rho}\left(\int_{0}^{\infty}|[b,\phi_k(t^{-1}\sqrt{H})]f|^2
\frac{\mathrm{d}t}{t}\right)^{1/2}\right\|_{L^2(\mathbb{R}^n,(1+|x|)^{-\alpha})}\nonumber\\
&\leq C_{\alpha,\upsilon}\|b\|_{\mathrm{BMO}}\sum_{k\geq1}2^{-k(\lambda-\frac{\alpha-1}{4}-\eta-\upsilon)}\|f\|_{L^2(\mathbb{R}^n,(1+|x|)^{-\alpha})}.
\end{align}
Then the $\mathrm{RHS}$ of \eqref{4f} is bounded by $C_{\alpha,\upsilon}\|b\|_{\mathrm{BMO}}\left\|f\right\|_{L^2(\mathbb{R}^n,(1+|x|)^{-\alpha})}$ provided  $\eta,\upsilon$ are small enough. From estimates~\eqref{4d} and \eqref{4g}, we can conclude that $\sup_{R>0}|[b,S_R^{\lambda}(H)]f|$ is bounded on $L^2(\mathbb{R}^n,(1+|x|)^{-\alpha})$ for $1< \alpha<n$ and $\lambda>(\alpha-1)/4$ whenever $\lambda>(\alpha-1)/4$. Combining $\sup_{R>0}|[b,S_R^{\lambda}(H)]f|$ is bounded on $L^2(\RR^n)$, we now use the complex interpolation method to obtain that $\sup_{R>0}|[b,S_R^{\lambda}(H)]f|$  is bounded on the
weighted space $L^2(\mathbb{R}^n,(1+|x|)^{-\alpha})$  for $0<\alpha\leq1$ when $n\geq2$. Firstly, we deal with $\sup_{R>0}|[b,S_R^{\lambda}(H)]f|$ by Kolmogorov-Seliverstov-Plessner linearization (see \cite[p.280]{St2}). In fact, we define $\mathcal{K}$ the nonnegative measurable function  on $\mathbb{R}^n$ which have only finite number of distinct values. Let $\mathcal{R}(x)\in\mathcal{K}$. It's claim that
$$\sup_{\mathcal{R}(x)\in\mathcal{K}}\|[b,S_{\mathcal{R}(x)}^{\lambda}(H)]f\|_{L^2(\mathbb{R}^n,(1+|x|)^{-\alpha})}= \|\sup_{R>0}|[b,S_R^{\lambda}(H)]f|\,\|_{L^2(\mathbb{R}^n,(1+|x|)^{-\alpha})}.$$
It's easy to see that
$$|[b,S_{\mathcal{R}(x)}^{\lambda}(H)]f(x)|\leq \sup_{R>0}|[b,S_R^{\lambda}(H)]f(x)|.$$
The converse inequality can be seen from the fact that we can always choose a series of function $\mathcal{R}_j(x)\in\mathcal{K}$ such that
$$\lim_{j\rightarrow+\infty}|[b,S_{\mathcal{R}_j(x)}^{\lambda}(H)]f(x)|= \sup_{R>0}|[b,S_R^{\lambda}(H)]f(x)|,\, \,\forall x\in\mathbb{R}^n.$$
 By Lebesgue's  dominated convergence theorem,
\begin{align*}
\|\sup_{R>0}|[b,S_R^{\lambda}(H)]f|\ \|_{L^2(\mathbb{R}^n,(1+|x|)^{-\alpha})}&=
\lim_{j\rightarrow+\infty}\|\ [b,S_{\mathcal{R}_j(x)}^{\lambda}(H)]f\|_{L^2(\mathbb{R}^n,(1+|x|)^{-\alpha})}\\
&\leq \sup_{\mathcal{R}(x)\in\mathcal{K}}\|\ [b,S_{\mathcal{R}(x)}^{\lambda}(H)]f\|_{L^2(\mathbb{R}^n,(1+|x|)^{-\alpha})}.
\end{align*}
Therefore, the boundeness of the sub-linear operator $\sup_{R>0}|[b,S_R^{\lambda}(H)]f|$ on $L^2(\mathbb{R}^n,(1+|x|)^{-\alpha})$ is equivalent to
$$\|[b,S_{\mathcal{R}(x)}^{\lambda}(H)]f\|_{L^2(\mathbb{R}^n,(1+|x|)^{-\alpha})}
\leq C\|f\|_{L^2(\mathbb{R}^n,(1+|x|)^{-\alpha})},$$
where the constant $C$ is independent of $\mathcal{R}(x)$. Then the weighted $L^2$ estimate of $[b,S_{\mathcal{R}(x)}^{\lambda}(H)]$ for $0<\alpha\leq1$ when $\lambda>0$ can be deduced from the complex interpolation theorem (see \cite[Theorem~2.11]{St3}).   As a consequence, we obtain that $\sup_{R>0}|[b,S_R^{\lambda}(H)]|$ is bounded on ${L^2(\mathbb{R}^n,(1+|x|)^{-\alpha})}$ for $0<\alpha\leq1$ if $\lambda>0$.
The proof of Theorem~\ref{thm1.1} is complete.
\end{proof}

\medskip

Finally, we give a proof of Corollary~\ref{coro1.2}.

\begin{proof}
For any $f\in C_c^\infty(\RR^n)$ and supp$f\subseteq B$ for some ball $B\subseteq \RR^n$, then $(b-b_B)f\in L^2(\RR^n)$ and $[b,S^{\lambda}_R(H)]f$ is well defined for any $\lambda>0$. By the almost everywhere convergence of Bochner-Riesz operator $S_R^\lambda(H)$(see \cite[Theorem~1.1]{CDHLY}), we have
$$\lim_{R\rightarrow\infty}[b,S^{\lambda}_R(H)]f(x)
=\lim_{R\rightarrow\infty}b(x)S^{\lambda}_R(H)(f)(x)-S^{\lambda}_R(H)(bf)(x)=0,\ \  \mathrm{a.e.} \ x\in\RR^n.
$$
Let $0\leq\alpha<n$. Theorem~\ref{thm1.1} implies that the almost everywhere convergence of $[b,S^{\lambda}_R(H)]f$ for any $f\in L^2(\RR^n,(1+|x|)^{-\alpha})$ if $\lambda>\max\{(\alpha-1)/4,0\}$.

For given $p\geq2$ and $\lambda>\lambda(p)/2$, we can choose $\alpha$ such that $n(1-2/p)<\alpha<4\lambda+1$ and $L^p(\mathbb{R}^n)\subseteq L^2(\mathbb{R}^n,(1+|x|)^{-\alpha})$.
Hence, the almost everywhere convergence of $[b,S^{\lambda}_R(H)]f$ holds for all $f\in L^p(\RR^n)$.
\end{proof}


\bigskip

\noindent
{\bf Acknowledgements}:
The authors would like to thank L. Yan for  helpful comments and suggestions.
P. Chen and X. Lin were supported by National Key R\&D Program of China 2022YFA1005702. P. Chen was supported by NNSF of China 12171489, Guangdong Natural Science Foundation 2022A1515011157.

\end{document}